\documentclass[draft]{amsart}
\usepackage{amssymb,graphicx,enumerate,amsthm}
\usepackage{multirow}
\usepackage{xcolor}
\usepackage{cite}
\usepackage{pgf,tikz,pgfplots}

\newcommand{\U}{\mathcal{U}}
\newcommand{\D}{\mathcal{D}}

\newcommand{\UG}{\mathcal{U}\mathcal{G}}
\newcommand{\GG}{\mathcal{G}\mathcal{G}}
\newcommand{\lG}{\ell\mathcal{G}}
\renewcommand{\L}{\mathcal{L}}
\newcommand{\R}{\mathcal{R}}
\renewcommand{\P}{\mathcal{P}}
\renewcommand{\S}{\mathcal{S}}

\newcommand{\Z}{\mathbb{Z}}

\numberwithin{equation}{section}
\newtheorem{theorem}{Theorem}[section]

\newtheorem{corollary}[theorem]{Corollary}

\begin{document}

\makeatletter
\def\imod#1{\allowbreak\mkern10mu({\operator@font mod}\,\,#1)}
\makeatother

\author{Ali Kemal Uncu}
   \address{University of Bath, Faculty of Science, Department of Computer Science, Bath, BA2 7AY, U.K.}
   \email{aku21@bath.ac.uk}
   \thanks{The research was partly funded by the Austrian Science Fund (FWF) grant numbers SFB50-07 and SFB50-09, and partly by the EPSRC Grant EP/T015713/1.}

\title[\scalebox{.9}{On double sum generating functions in connection with some classical partition theorems}]{On double sum generating functions in connection with some classical partition theorems}
     
\begin{abstract} We focus on writing double sum representations of the generating functions for the number of partitions satisfying some gap conditions. Some example sets of partitions to be considered are partitions into distinct parts and partitions that satisfy the gap conditions of the Rogers--Ramanujan, G\"ollnitz--Gordon, and little G\"ollnitz theorems. We refine our representations by imposing a bound on the largest part and find finite analogues of these new representations. These refinements lead to many $q$-series and polynomial identities. Additionally, we present a different construction and a double sum representation for the products similar to the ones that appear in the Rogers--Ramanujan identities.
\end{abstract}

  


\date{\today}
   
\dedicatory{To the one and only Alexander Berkovich}
   
\maketitle

\section{Introduction}\label{Sec_intro}

Let a partition $\pi$ be a \textit{non-decreasing finite sequence} $(\lambda_1,\lambda_2,\dots)$ of positive integers. The $\lambda_i$ are called parts of the partition $\pi$. We denote the number of parts of $\pi$ with $\#(\pi)$. The sum $\lambda_1+\lambda_2+\dots+\lambda_{\#(\pi)}$ is called the size of $\pi$. Conventionally, we define the empty sequence as the only partition of 0. 

In the recent years, there have been many advancements in representing generating functions for the number of partitions that satisfy some gap conditions as multiple sums \cite{Alladi_Berkovich,Alladi_Berkovich2, BerkovichUncu7, Kagan1, Kagan2, Kagan3, Kanade_Russell, Andrews_Mahlburg_Bringmann, sills}. One common theme most of these works share is that their representations all rely on some combinatorial construction. This is an invaluable asset, as an explicit combinatorial construction may be adapted and exploited for other situations. Here we would like to present the use of the ideas coming from \cite{Kanade_Russell, Kagan1, Kagan2, Kagan3} to many fundamental ``gap--congruence" partition theorems, such as Euler's partition, Rogers--Ramanujan, G\"ollnitz--Gordon and the little G\"ollnitz theorems. We will also refine the construction using the ideas presented in \cite{BerkovichUncu7}, and compare the refinements we get with other known refinements of such generating functions related to these classical results. This will not only give us new double sum representations for generating functions for the number of partitions with gap conditions, but also yield intriguing polynomial and $q$-series identities, such as the following two theorems.

\begin{theorem}\label{Fin_GG_Ali_THM} For any non-negative integer $N$, Let 
\begin{align}
\label{GG1_min_config_size}G (m,n):= \frac{3m^2+m}{2}+ 2mn+n^2.
\intertext{We have }
\label{Fin_GG1_Intro_EQN}\sum_{m,n\geq 0}q^{G(m,n)}{2N-2m-2n+2 \brack m}_q{N-m-n+1 +\lfloor \frac{n}{2}\rfloor \brack \lfloor \frac{n}{2}\rfloor}_{q^4} &=\sum_{m,n\geq 0} q^{m^2+n^2} {N-m+1\brack n}_{q^2}{n\brack m}_{q^2},\\
\label{Fin_GG2_Intro_EQN}\sum_{m,n\geq 0} q^{G(m,n)+2m+2n}{2N-2m-2n \brack m}_q{N-m-n+\lfloor \frac{n}{2}\rfloor \brack \lfloor \frac{n}{2}\rfloor }_{q^4} &= \sum_{m,n\geq 0} q^{m^2+n^2+ 2n} {N-m\brack n}_{q^2}{n\brack m}_{q^2},
\end{align} where $\lfloor x \rfloor$ is the greatest integer $\leq x$.
\end{theorem}

Here and in the rest of the paper we use the standard $q$-series notations \cite{Theory_of_Partitions,Gasper_Rahman}. For variables $a_i$ and $q$ (with $|q|<1$), and a non-negative integer $n$
\begin{align*}
&(a)_n:=(a;q)_n = \prod_{k=0}^{n-1}(1-aq^k),\text{   and   }  (a;q)_\infty:= \lim_{n\rightarrow\infty}(a;q)_n,\\ 
(a_1,&a_2,\dots,a_k;q)_N := (a_1;q)_n(a_2;q)_n\dots (a_k;q)_n\text{ for any }n\in\Z_{\geq0}\cup\{\infty\}.	
\end{align*}
We define the $q$-binomial coefficients as
\begin{equation}
 \label{Binom_def}  \displaystyle {m+n \brack m}_q := \left\lbrace \begin{array}{ll}\frac{(q)_{m+n}}{(q)_m(q)_{n}}&\text{for }m, n \geq 0,\\   0&\text{otherwise.}\end{array}\right.
\end{equation}

\begin{theorem}\label{Intro_THM_little_G} We have
\begin{equation}\label{little_Gollnitz_Difference}
\sum_{m,n\geq 0} \frac{q^{3\frac{m^2+m}{2} + 2nm + n^2 +3n+2m+1}}{(q;q)_m(q^4;q^4)_{\lfloor n/2 \rfloor}} = \frac{1}{(q,q^5,q^6;q^8)_\infty}-\frac{1}{(q^2,q^3,q^7;q^8)_\infty}.
\end{equation}
\end{theorem}

We shall note that the Theorem~\ref{Intro_THM_little_G} is similar to \cite[Thm~1.2, p. 2]{BerkovichUncu8}.

In this paper, we also present a new combinatorial construction to find new representations of the generating functions with congruence conditions. We will prove the following double series representations of the Rogers--Ramanujan identities related products. 

\begin{theorem}\label{RRprod_intro_thm} We have
\begin{align}
\label{Prod_GF_RR1}\frac{1}{(xq,xq^4;q^5)_\infty} &= \sum_{j\geq0} \frac{x^j q^{j}}{(q^{10},q^{10})_{\lfloor j/2 \rfloor}} \sum_{i=0}^{\lfloor j/2 \rfloor} { \lfloor j/2\rfloor \brack i}_{q^{10}} \frac{(-q^3;q^{10})_{i+\lceil j/2\rceil -\lfloor j/2 \rfloor} (-q^{8};q^{10})_i}{(q^5;q^{10})_{i+\lceil j/2\rceil -\lfloor j/2 \rfloor}} (q^{6})^{\lfloor j/2\rfloor -i},\\
\label{Prod_GF_RR2}\frac{1}{(xq^2,xq^3;q^5)_\infty} &= \sum_{j\geq0} \frac{x^j q^{2j}}{(q^{10},q^{10})_{\lfloor j/2 \rfloor}} \sum_{i=0}^{\lfloor j/2 \rfloor} { \lfloor j/2\rfloor \brack i}_{q^{10}} \frac{(-q;q^{10})_{i+\lceil j/2\rceil -\lfloor j/2 \rfloor} (-q^{6};q^{10})_i}{(q^5;q^{10})_{i+\lceil j/2\rceil -\lfloor j/2 \rfloor}} (q^{2})^{\lfloor j/2\rfloor -i},
\end{align} where $\lceil x\rceil$ is the smallest integer $\geq x$.
\end{theorem}

This work is organized as follows. In Section~\ref{Sec_Uniform_gap} we present new double sum representations of the generating functions of partitions with uniform gap conditions using constructive ideas similar to \cite{BerkovichUncu7, Kagan1, Kagan2}. This includes new representations for the generating functions for the number of unrestricted partitions, partitions into distinct parts, and partitions with Rogers--Ramanujan gap conditions and their polynomial analogues. We write new double sum representations for the generating functions related to G\"ollnitz--Gordon and little G\"ollnitz identities in Sections~\ref{Sec_GG} and \ref{Sec_lG}, respectively. We also discuss finite analogs of these double sums, and prove some polynomial identities related to these finitizations such as Theorem~\ref{Fin_GG_Ali_THM}. We give a double sum representation for the products of type $(q^a,q^{b};q^M)^{-1}_\infty$ in Section~\ref{Sec_Products}. Section~\ref{Sec_Outlook} has a brief summary of the implications of this result and some results the author wants to address in the future.

\section{Partitions with uniform gap conditions, double sum generating functions and their finitizations}\label{Sec_Uniform_gap}

Let $\U$ be the set of all partitions and $\D$ be the set of partitions into distinct parts. Also define
\begin{equation}\label{Def_UG_set}\UG_{k,l} := \{\pi = (\lambda_1,\lambda_2,\dots ,\lambda_{\#(\pi)}) \in \U\ : \lambda_1\geq l,\  \lambda_{i} - \lambda_{i-1} \geq k\text{ for }2\leq i\leq \#(\pi) \},\end{equation} for any $k\in\Z_{\geq0}$ and positive integer $l$. Partitions from the set $\UG_{k,l}$ have \textit{uniform gap conditions} such that the gap between consecutive parts is greater or equal to $k$. It is easy to observe that $\U = \UG_{0,1}$ and $\D = \UG_{1,1}$. 

The partitions with uniform gap conditions are well studied, and appear in many beautiful classical theorems. We would like to remind the reader the most famous results are Euler's Partition Theorem and Rogers--Ramanujan identities.

\begin{theorem}[Euler's Partition Theorem]\label{Euler_Ptt_THM_Combin}
\[\sum_{\pi\in\UG_{1,1}} q^{|\pi|}  = \sum_{\pi\in\mathcal{O}} q^{|\pi|},\] where $\mathcal{O}$ is the set of partitions into odd parts.
\end{theorem}

Let $\mathcal{C}_{i,5}$ be the set of partitions with parts congruent to $\pm i$ modulo 5, for $i=1$ and 2. Then,

\begin{theorem}[Rogers--Ramanujan identities]\label{RR_THM_Combin} For $i=1$ and $2$,
\[\sum_{\pi\in\UG_{2,i}} q^{|\pi|}  = \sum_{\pi\in\mathcal{C}_{i,5}} q^{|\pi|}.\]
\end{theorem}

These combinatorial theorems also have their $q$-series counterpart:
\begin{align}
\label{Euler_analytic} \sum_{j\geq 0} \frac{q^{j(j+1)/2}}{(q;q)_j} &= \frac{1}{(q;q^2)_\infty},\\
\label{RR_analytic} \sum_{j\geq 0} \frac{q^{j^2+(i-1)j}}{(q;q)_j} &= \frac{1}{(q^i,q^{5-i};q^5)_\infty},\text{ for }i\in\{1,2\},
\end{align} respectively.

In general, we have \begin{equation}\label{Singlesum_UF_GF_inf}
\sum_{\pi\in\UG_{k,l}} x^{\#(\pi)}q^{|\pi|} = \sum_{j\geq 0} \frac{x^j q^{lj + kj(j-1)/2}}{(q;q)_j}.
\end{equation} The right-hand side summand of \eqref{Singlesum_UF_GF_inf} can be interpreted as the generating function for partitions from $\UG_{k,l}$ and the left-hand side summands of \eqref{Euler_analytic} and \eqref{RR_analytic} are special cases of \eqref{Singlesum_UF_GF_inf}. This interpretation comes from the following argument. One observes that $lj + kj(j-1)/2$ is the size of the partition \[\pi_1 = (l,l+k, l+2k, \dots,l+(j-2)k, l+ (j-1)k),\] which has exactly $j$ parts each $\geq l$, each $k$-distant to its neighboring parts. It is well known that, the $(q;q)^{-1}_j$ is the generating function of partitions $\pi_2 $ into $\leq j$ parts. By adding the parts of $\pi_2$ to the parts of $\pi_1$ starting from the largest part, we get a partition $\pi_3\in\UG_{k,l}$ into exactly $j$ parts. This operation is bijective, and this shows that the summand of the right-hand side of \eqref{Singlesum_UF_GF_inf} is the generating function for the partitions from $\UG_{k,l}$ into exactly $j$ parts. Summing this summand over all possible number of parts ($j\geq 0$) shows that the right-hand side \eqref{Singlesum_UF_GF_inf} is a closed form of the generating function on the left-hand side of \eqref{Singlesum_UF_GF_inf}.

Now, we would like to move on to writing a new closed form for the generating functions, \begin{equation}\label{GenFunc_UG_abs}\sum_{\pi\in\UG_{k,l}}q^{|\pi|},\end{equation} using double sums, inspired by the recent works \cite{Kagan1, Kagan2, BerkovichUncu7}. We define the partition 
\begin{equation}\label{minimal_config_pttn}{}_{m,n}\pi_{k,l} :=(\underbrace{l,l+k,l+2k,\dots,l+(n-1)k}, l+(n-1)k + (k+1),\dots,l+(n-1)k + m(k+1) ),\end{equation} for some positive integer $l$, and non-negative integers $k,\ m,$ and $n$. We define ${}_{0,0}\pi_{k,l}$ to be the empty partition with size 0. 

We call the underbraced portion of the partition its \textit{initial chain} of length $n$. Note that the partition ${}_{m,n}\pi_{k,l}$ is the partition with $n$ parts that have the minimal gap condition $k$, and with $m$ parts (called \textit{singletons}) that have a gap of $k+1$ with their neighbouring parts. Moreover, ${}_{m,n}\pi_{k,l}$ is the partition with the smallest size: \begin{equation}\label{UG_min_config_size} Q_{k,l}(m,n):=(m+n)l+k\frac{(m+n)(m+n-1)}{2}+\frac{m(m+1)}{2} ,\end{equation} which has an initial chain of length $n$ and $m$ singletons all $>l$. We will call such a partition a \textit{minimal configuration}. 
 
Repeating \cite{BerkovichUncu7, Kagan1,Kagan2}, for any fixed pair $(m,n)\in\Z^2_{\geq 0}$, our idea is to start with the minimal configuration, define reversible \textit{rules of motion}, and construct other partitions as the descendants of ${}_{m,n}\pi_{k,l}$. After this combinatorial study we will be able to write a new closed form formula for \eqref{GenFunc_UG_abs} by adding together all possible generating functions related to $(m,n)$ pairs.

For a fixed pair $(m,n)\in\Z^2_{\geq 0}$, we now define the rules of motion. 

\begin{enumerate}[i.]
\item Motion of the singletons:\vspace{.2cm}

Assuming $m$ is positive, one can add any positive integer amount $r_m$ to the largest part, $l+(n-1)k + m(k+1)$, of ${}_{m,n}\pi_{k,l}$ and this creates a new partition, which still has an intact $n$-length initial chain and $m$-parts with larger gaps between parts. We will refer to the addition of the $r_m$ value to the part $l+(n-1)k + m(k+1)$ as \textit{moving forward the part $l+(n-1)k + m(k+1)$ by $r_m$}. This forward motion is reversible since one can easily recover $r_m$ from the outcome partition by moving the largest part backwards to create ${}_{m,n}\pi_{k,l}$. Then, one can move the second largest part, after the largest part, again reversibly with some $r_{m-1} \leq r_m$. The outcome partition would still have the same $n$-length initial chain as  ${}_{m,n}\pi_{k,l}$ and $m$ singletons with gaps $\geq k+1$. In this manner, we can repeat this process for smaller singletons. This way one can reversibly move the $m$ singletons forward by starting from the largest part and moving it by $r_m \geq 0$ steps forward, then the second largest and moving that part $r_{m-1}\geq 0$ steps forward (for any $r_{m-1}$ with $r_{m-1}\leq r_m$), then move the third largest part $r_{m-2}\geq 0$ steps forward (where $r_{m-2}\leq r_{m-1}$), etc. After these forward motions, the outcome partition would still have the same $n$-length initial chain with ${}_{m,n}\pi_{k,l}$ and it will have $m$ scattered singletons with gaps $\geq k+1$ with their consecutive parts. 

We would like to further emphasize the reversibility of these forward motions. Let $pi\in \UG_{k,l}$ with an intact $n$-length initial chain $l,l+k, l+2k, \dots,l+(n-2)k, l+ (n-1)k$ followed by $m$ singletons $\lambda_1,\lambda_2,\dots, \lambda_m$ satisfying $l+(n-1)k+(k+1)\leq \lambda_1 $ and $\lambda_i - \lambda_{i-1} \geq k+1$ for $i=2,3,\dots, m$. Then, $r_i$'s are uniquely determined by $r_i := \lambda_i - (l + (n-1)k+i(k+1))$ for $i=1,2,\dots,m$. We have $r_i \geq r_{i-1}\geq 0$ since $\lambda_i \geq \lambda_{i-1} + (k+1)$ for all $i=2,\dots, m$. By moving $\lambda_i$ backwards $r_i$-steps for all $i=1,2,\dots,m$, we get exactly the minimal configuration ${}_{m,n}\pi_{k,l}$.

The sequence \[(r_1,r_2,\dots,r_{m-1},r_m)\] related to such motions is a non-decreasing finite sequence, where some $r_j$'s might be equal to zero. It is clear that these sequences are in bijection with partitions into $\leq m$ parts. In the light of this bijection we conclude that the generating function for the size of the sequences that define reversible forward motions for the $m$ singletons of ${}_{m,n}\pi_{k,l}$ is the same as the generating function for the number of partitions into $\leq m$ parts, which is \begin{equation}\label{singleton_move_GF}\frac{1}{(q;q)_m}.\end{equation}

\item Forward motion rules starting from the initial chain and crossing singletons:\vspace{.2cm}

We will be moving two parts one step forwards together at once from the end of the initial chain. This way a forward motion will be adding 2 to the total size of the partition. That being said, if $n$ is an odd number, the smallest part $l$ will not be moving. 

Given an initial chain, \[\underbrace{l,l+k,l+2k,\dots, l+(n-2)k,l+(n-1)k},\] we start moving terms forward by splitting the last two parts and moving them one step forward each:  \begin{align}\nonumber&\underbrace{l,l+k,l+2k,\dots, l+(n-2)k,l+(n-1)k}\\\label{initial_chain_split}&\hspace{2cm}\mapsto\underbrace{l,l+k,l+2k,\dots l+(n-3)k}, \underline{l+(n-2)k+1,l+(n-1)k+1}.\end{align} The last two underlined parts with $k$ difference between them is called a \textit{pair} and the smaller part in this pair is in $k+1$ distance to the new initial chain. Note that \eqref{initial_chain_split} is a reversible motion. A pair has $k$ difference between its parts which distinguishes its parts from any singleton that the partition might have. 

One can move a pair forward freely (when permisible) by adding one to both its parts: \begin{equation}
\label{pair_free_move} \underline{x,x+k}\mapsto \underline{x+1,x+k+1}.
\end{equation} While moving a pair forward using~\eqref{pair_free_move}, a pair $\underline{x,x+k}$ may come in $k$ distance to a singleton $x+2k$. Then to move this pair forward, we need to define a crossing over singletons rule of pairs. This can be done as follows: \begin{equation}
\label{pairs_crossing} \underline{x,x+k}\, ,\ x+2k \mapsto x,\ \underline{x+k+1,x+2k+1}.
\end{equation}
There is at least $k+1$ difference between the singleton $ x+2k$ and the closest larger part after it. Hence, after employing \eqref{pairs_crossing}, we still end up with a partition in $\UG_{k,l}$.

The defined motions \eqref{initial_chain_split}, \eqref{pair_free_move}, \eqref{pairs_crossing} are all reversible and they each add 2 to the total size of a partition.

Therefore, one can start from the end of a given chain, split a pair and move that pair forward according to the rules, then go back to the end of the initial chain split another pair and move that pair forward (less than or equal times to the previous pair) and repeat these steps. Just as in the singletons' case, we move the first pair the longest distance, the second pair less than or equal to that distance, etc. This way we never needed to handle pairs crossing over other pairs. Similar to the singletons' case, we can relate the splitting of the pairs from the initial chain and these pairs' forward motions with a partition into $\leq \lfloor n/2\rfloor$ parts, as there can at most be $ \lfloor n/2\rfloor$ pairs, where $\lfloor \cdot \rfloor $ is the standard floor function. We write the related generating function as \begin{equation}\label{pair_movement_GF}\frac{1}{(q^2;q^2)_{\lfloor n/2 \rfloor}}.\end{equation} This is analogous to the previous case. Since every motion adds 2 to the total size of the overall partition we use $q^2$ in the place of $q$ in our generating functions.
\end{enumerate}

All the motions are reversible and for any given partition $\pi$ from $\UG_{k,l}$ one can now find the unique minimal configuration ${}_{m,n}\pi_{k,l}$ that $\pi$ is descends from. Moreover, one can determine $m$ and $n$ essentially effortlessly. 

For example, if we start with the partition \[ (3,5,7,13,16,18,20,22,24,28) \in \UG_{2,3},\] we can directly identify the initial chain by drawing a brace under the parts starting from $l=3$ that are $k=2$ distant to each other: \[ (\underbrace{3,5,7},13,16,18,20,22,24,28).\] Then, we identify all the pairs by examining the parts of the partition in an increasing order and grouping $2$-distant parts with each other: \[ (\underbrace{3,5,7},13,\underline{16,18},\underline{20,22},24,28).\] All the other parts of the partition are singletons and the number of those parts is $m$. The total number of parts in the initial chain and the pairs is $n$. For this example, $m=3$ and $n=7$. To revert this partition back to the minimal configuration \[{}_{3,7}\pi_{2,3} = (\underbrace{3,5,7,9,11,13,15},18,21,24),\] we start by moving the smallest pair $\underline{16,18}$ backwards using the rules \eqref{initial_chain_split}, \eqref{pair_free_move} and \eqref{pairs_crossing} in reverse. In particular, for the pair $\underline{16,18}$ we first employ \eqref{pairs_crossing} in reverse to get \[(\underbrace{3,5,7},\underline{13,15},17,\underline{20,22},24,28).\] This is then followed by applying \eqref{pair_free_move} backwards three times followed by \eqref{initial_chain_split} in reverse to get \[(\underbrace{3,5,7,9,11},17,\underline{20,22},24,28).\] This is followed by moving the new smallest pair $\underline{20,22}$ backwards. Once the backwards motions of all the pairs are finished and these pairs are welded to the initial chain, we move the singletons backwards one-by-one starting from the smallest to the largest.

Putting \eqref{UG_min_config_size}, \eqref{singleton_move_GF}, and \eqref{pair_movement_GF} together, we see that the generating function for the descendant of ${}_{m,n}\pi_{k,l}$ is \begin{equation}\label{GF_fixed_m_n}
\frac{q^{Q_{k,l}(m,n)}}{(q;q)_m(q^2;q^2)_{\lfloor n/2\rfloor}}.
\end{equation} One can directly include $x^{m+n}$ to \eqref{GF_fixed_m_n} and count the number of parts using the variable $x$ in this generating function.
Summing over all $(m,n)\in\mathbb{Z}^2_{\geq0}$, we get a closed form for the generating function of partitions from $\UG_{k,l}$ grouped with respect to their sizes.
\begin{theorem}\label{Unif_Gap_THM_inf} For any non-negative integer $k$ and positive integer $l$ we have
\begin{equation}\label{GF_UG_Double_sum_with_floor}\sum_{\pi\in\UG_{k,l}}x^{\#(\pi)}q^{|\pi|} = \sum_{m,n\geq 0}\frac{x^{m+n}q^{Q_{k,l}(m,n)}}{(q;q)_m(q^2;q^2)_{\lfloor n/2\rfloor}},\end{equation} where $Q_{k,l}(m,n)$ is defined in \eqref{UG_min_config_size}. Equivalently, one can split the even and odd cases for $n$ and write the following by combining the outcome without the presence of a floor function:
\begin{equation}\label{GF_UG_Double_sum}
\sum_{\pi\in\UG_{k,l}}x^{\#(\pi)}q^{|\pi|} = \sum_{m,n\geq 0}\frac{x^{m+2n}q^{Q_{k,l}(m,2n)}(1+xq^{l+(m+2n)k})}{(q;q)_m(q^2;q^2)_{n}}.
\end{equation}
\end{theorem}

Some corollaries of this theorem are as follows:
\begin{corollary}\label{UG_Corollary}
\begin{align}
\label{Doublesum_GF_U}\sum_{\pi\in\UG_{0,1}}x^{\#(\pi)} q^{|\pi|} &= \sum_{m,n\geq 0} \frac{x^{m+2n}q^{\frac{m(m+3)}{2}+2n}(1+xq)}{(q;q)_m(q^2;q^2)_n},\\
\label{Doublesum_GF_D}\sum_{\pi\in\UG_{1,1}}x^{\#(\pi)} q^{|\pi|} &= \sum_{m,n\geq 0} \frac{x^{m+2n}q^{m^2+m+2mn+2n^2+n}(1+xq^{m+2n+1})}{(q;q)_m(q^2;q^2)_n},\\
\label{Doublesum_GF_RR1}\sum_{\pi\in\UG_{2,1}}x^{\#(\pi)} q^{|\pi|} &= \sum_{m,n\geq 0} \frac{x^{m+2n}q^{\frac{m(3m+1)}{2}+4mn+4n^2}(1+xq^{2m+4n+1})}{(q;q)_m(q^2;q^2)_n},\\
\label{Doublesum_GF_RR2}\sum_{\pi\in\UG_{2,2}}x^{\#(\pi)} q^{|\pi|} &= \sum_{m,n\geq 0} \frac{x^{m+2n}q^{\frac{m(3m+1)}{2} + 4mn +4n^2 + m+2n}(1+xq^{2m+4n+2})}{(q;q)_m(q^2;q^2)_n}.
\end{align}
\end{corollary}

We would like to point out that comparing right-hand sides of \eqref{Singlesum_UF_GF_inf} and \eqref{GF_UG_Double_sum_with_floor} gives the following theorem.
\begin{theorem}\label{Single_sum_to_double_sum_THM} \begin{equation}\label{Double_to_single} \sum_{j\geq 0} \frac{x^j q^{lj + kj(j-1)/2}}{(q;q)_j}= \sum_{m,n\geq 0}\frac{x^{m+n}q^{Q_{k,l}(m,n)}}{(q;q)_m(q^2;q^2)_{\lfloor n/2\rfloor}},\end{equation} where $Q_{k,l}(m,n)$ is defined as in \eqref{UG_min_config_size}.\end{theorem}

We can also prove \eqref{Double_to_single} using basic hypergeometric series techniques. 
\begin{proof}Let $M=m+n$, after simplifying the $q$-Pochhammer symbols using \cite[p.351, I.10]{Gasper_Rahman}, we get \begin{equation}\label{Double_to_single_sum_M} \sum_{m,n\geq 0}\frac{x^{m+n}q^{Q_{k,l}(m,n)}}{(q;q)_m(q^2;q^2)_{\lfloor n/2\rfloor}} = \sum_{M=0}^\infty \frac{x^M q^{lM + kM(M-1)/2+M(M+1)/2}}{(q;q)_M} \sum_{n\geq 0} \frac{(q^{-M};q)_n}{(q^2;q^2)_{\lfloor n/2 \rfloor}}(-1)^n.\end{equation} Now we focus on the inner sum on the right-hand side. We first split the even and odd cases of $n$, and rewrite the $q$-Pochhammer symbols in the same base: \begin{align*}\sum_{n\geq 0}\frac{(q^{-M};q)_n}{(q^2;q^2)_{\lfloor n/2 \rfloor}}(-1)^n &= \sum_{n \geq 0} \frac{(q^{-M}; q)_{2n}}{(q^2;q^2)_n} - \sum_{n\geq 0} \frac{(q^{-M};q)_{2n+1}}{(q^2;q^2)_{n}}, \\ &= \sum_{n \geq 0} \frac{(q^{-M}, q^{-M+1};q^2)_{n}}{(q^2;q^2)_n} - \sum_{n\geq 0} \frac{(q^{-M}, q^{-M+1};q^2)_{n} (1-q^{-M+2n})}{(q^2;q^2)_{n}}.\\
\intertext{Then, it is easy to see that this expression reduces to} &= q^{-M} \sum_{n \geq 0} \frac{(q^{-M}, q^{-M+1};q^2)_{n}}{(q^2;q^2)_n} q^{2n}.
\end{align*} This terminating series can be summed using the $q$-Chu--Vandermonde identity \cite[p.354, II.6]{Gasper_Rahman} \[\sum_{i\geq 0} \frac{(q^{-n},a;q)_i}{(q,b;q)_i}q^j = \frac{(b/a;q)_n}{(b;q)_n}a^n\] with $(a,n,b,q) = (q^{-M+1},M/2,0,q^2)$ if $M$ is even, and $(a,n,b,q) = (q^{-M},(M-1)/2,0,q^2)$ otherwise. This shows that \begin{equation}\label{ChuChu} \sum_{n\geq 0} \frac{(q^{-M};q)_n}{(q^2;q^2)_{\lfloor n/2 \rfloor}}(-1)^n = q^{-M(M+1)/2}.\end{equation} Substituting \eqref{ChuChu} in \eqref{Double_to_single_sum_M} yields \eqref{Double_to_single}.
\end{proof}

Now we will move onto putting a bound on the largest part of the partitions to be counted and the polynomial analog of Theorem~\ref{Single_sum_to_double_sum_THM}. For any non-negative integer $N$, let $\UG_{k,l,N}$ be the set of partitions from $\UG_{k,l}$ with the extra constraint that all the parts are $\leq N$. In the spirit of \eqref{Singlesum_UF_GF_inf}, it is easy to see that the generating function for the number of partitions is \begin{equation}\label{UG_bounded_single_sum} \sum_{\pi\in\UG_{k,l,N}} x^{\#(\pi)}q^{|\pi|}= \sum_{j \geq 0}x^j q^{lj + kj(j-1)/2} {N-l-(j-1)k+j \brack j}_q + \chi(0\leq N<l-k),
\end{equation} where \begin{equation}\label{Chi_def} \chi(\text{``statement"})=\left\lbrace \begin{array}{ll}
1, & \text{if the statement is true,}\\
0, & \text{otherwise.}
\end{array} \right.\end{equation}In this finitization, we make sure that the $\pi_2$ (for some $j$) in the construction of \eqref{Singlesum_UF_GF_inf} has all its parts $\leq N-l-(j-1)k$ by replacing the $q$-Pochhammer $(q;q)^{-1}_j$ with the necessary $q$-binomial coefficient, which is the generating function for the number of partitions in a $j \times (N-l-(j-1)k)$ box. Moreover, the empty partition is an element of $\UG_{k,l,N}$ for any non-negative $N$. Empty partition does not get counted by the right-hand side sum of \eqref{UG_bounded_single_sum} if $0\leq N < l-k$. Therefore, we add the reflective correction term to our calculations.

The generating functions of Theorem~\ref{Unif_Gap_THM_inf} that are represented as double sums are also suitable for this type of finitization. We can find the finite versions of the generating function of Theorem~\ref{Unif_Gap_THM_inf} by studying the defined motions with care. 

We would like to generate all the partitions in $\UG_{k,l,N}$ for some non-negative integer $N$ using minimal configurations. In the unbounded case, given minimal configuration ${}_{m,n}\pi_{k,l}$, the free forward motion of the singletons corresponds to a partition into $\leq m$ parts, which has the generating function \eqref{singleton_move_GF}. If we require all the parts of the outcome partition to be $\leq N$, we can only move the largest singleton of ${}_{m,n}\pi_{k,l}$ to the upper bound $N$. Recall that the largest part of ${}_{m,n}\pi_{k,l}$ is $l+(n-1)k + m(k+1)$. This part can only move $N - (l+(n-1)k + m(k+1))$ steps forward before reaching the upper bound on the largest part $N$. Therefore, now the motion of the singletons correspond to partitions into $\leq m$ parts each $\leq N-(l+(n-1)k + m(k+1))$. In other words, the motion of the singletons corresponds to partitions that fit in a $m\times (N-(l+(n-1)k + m(k+1)))$ box. The corresponding generating function for these partitions is \begin{equation}\label{UG_singleton_bunded_move_GF}
{N-(l+(n-1)k + m(k+1)) + m \brack m}_q,
\end{equation} instead of $(q;q)^{-1}_m$ of \eqref{singleton_move_GF}.

Similarly, we can move the largest part $l+(n-1)k$ of the initial chain of ${}_{m,n}\pi_{k,l}$ forward till $N$. Observe that a pair crossing over a singleton, defined by the motion \eqref{pairs_crossing}, darts the pair $k$ steps forward. We need to include crossing over all $m$-singletons in our construction. It is easily seen that crossing over all $m$ singletons darts forward a pair an extra $km$ steps. Therefore, we see that the tail end $\underline{\dots,l+(n-2)k,l+(n-1)k}$ of the initial chain can move $N-(l+(n-1)k) - km$ steps forward before the largest part $l+(n-1)k$ reaches the upperbound $N$. Hence, the related generating function for the motion of all $\lfloor n/2 \rfloor$ pairs one splits from the initial chain of the minimal configuration ${}_{m,n}\pi_{k,l}$ is given by \begin{equation}\label{UG_pairs_bunded_move_GF}
{N-(l+(n-1)k) - km + \lfloor n/2 \rfloor \brack \lfloor n/2\rfloor}_{q^2}.
\end{equation}

Once again the empty partition is missed by the $q$-binomial coefficients if $N<l-k$. Hence, we need to include the same correction term as in \eqref{UG_bounded_single_sum}. Putting \eqref{UG_singleton_bunded_move_GF}, \eqref{UG_pairs_bunded_move_GF} and the correction term together we see get the following. 

\begin{theorem}\label{Unif_Gap_THM_Bounded} For any non-negative integers $N$ and $k$, and a positive integer $l$ we have
\begin{align}\label{GF_UG_Bounded_Double_sum_with_floor}\sum_{\pi\in\UG_{k,l,N}}&x^{\#(\pi)}q^{|\pi|}\\\nonumber &\hspace{-1cm}= \sum_{m,n\geq 0}x^{m+n}q^{Q_{k,l}(m,n)}{N-(l+(n-1)k + m(k+1)) + m \brack m}_q{N-(l+(n-1)k) - km + \lfloor n/2 \rfloor \brack \lfloor n/2\rfloor}_{q^2}\\\nonumber&\hspace{-1cm}+ \chi(0\leq N<l-k),\end{align} where $Q_{k,l}(m,n)$ is defined in \eqref{UG_min_config_size}, and $\chi$ is defined in \eqref{Chi_def}.
\end{theorem}

Comparing the right-hand sides of \eqref{UG_bounded_single_sum} and \eqref{GF_UG_Bounded_Double_sum_with_floor} yields the following analytic theorem.

\begin{theorem}\label{UG_polynomial_equality_THM} For any non-negative integers $N$ and $k$, and a positive integer $l$ we have
\begin{align}\label{UG_polynomial_equality}\sum_{j \geq 0} &x^j q^{lj + kj(j-1)/2} {N-l-(j-1)k+j \brack j}_q\\\nonumber &= \sum_{m,n\geq 0}x^{m+n}q^{Q_{k,l}(m,n)}{N-(l+(n-1)k + m(k+1)) + m \brack m}_q{N-(l+(n-1)k) - km + \lfloor n/2 \rfloor \brack \lfloor n/2\rfloor}_{q^2},\end{align} where $Q_{k,l}(m,n)$ is defined in \eqref{UG_min_config_size}.
\end{theorem}

Theorem~\ref{UG_polynomial_equality_THM} is combinatorially proven by the above construction, but we would also like to give a generating function proof using $q$-series techniques. Here we would also like to mention that the choice $N\geq l$ is conventional. We pick this convention only to make our initial conditions of our recurrences simpler.

\begin{proof} Observe that if $k$ is positive both sides of \eqref{UG_polynomial_equality} are polynomials. This is not true for the $k=0$ case and we want to treat this case separately from the rest. Let $k=0$, then the $q$-exponential summation \cite[p.354, II.2]{Gasper_Rahman} \[\sum_{i\geq 0 } \frac{q^{i(i-1)/2} z^i}{(q;q)_i} = (-z;q)_\infty,\] and even-odd split of the variable $n$ followed by the $q$-binomial theorem \cite[p.354, II.4]{Gasper_Rahman} \[\sum_{i\geq 0 } \frac{(a;q)_i}{(q;q)_i} z^i = \frac{(az;q)_\infty}{(z;q)_\infty}\] on the left-hand side yields
\begin{align*}
\sum_{j \geq 0} x^j q^{lj } {N-l+j \brack j}_q &= \frac{1}{(xq^l;q)_{N-l+1}},
\intertext{and}
\sum_{m,n\geq 0}x^{m+n}q^{(m+n)l+m(m+1)/2}{N-l \brack m}_q&{N-l+ \lfloor n/2 \rfloor \brack \lfloor n/2\rfloor}_{q^2} \\&= \sum_{m\geq 0} x^m q^{m(l+1)+m(m-1)/2} {N-l \brack m}_q \sum_{n\geq 0}x^n q^{nl}{N-l+ \lfloor n/2 \rfloor \brack \lfloor n/2\rfloor}_{q^2} \\
&= (-xq^{l+1};q)_{N-l} \left( (1+xq^l) \sum_{n\geq 0}x^{2n} q^{2nl}{N-l+ n \brack n}_{q^2}\right) \\
&= \frac{(-xq^l;q)_{N-l+1}}{(x^2 q^{2l};q^2)_{N-l+1}} = \frac{1}{(xq^l;q)_{N-l+1}}.
\end{align*} 

Now, assume that $N$ is non-negative, and the pair $k$ and $l$ are two fixed positive integers. Denote the left-hand side and right-hand side sums of \eqref{UG_polynomial_equality} by ${}_l\S(N,q)$ and ${}_r\S(N,q)$, respectively. Similarly, denote the summands of ${}_l\S(N,q)$ and ${}_r\S(N,q)$ by $\L_j(N,q)$ and $\R_{m,n}(N,q)$, respectively. 

Observe that $\L_j(N,q)$ satisfies the recurrence \begin{equation}\label{UG_F_recurrence}
\L_j(N,q)= \L_j(N-1,q) + xq^N \L_{j-1}(N-k,q) + \delta_{j,0}\,\delta_{ N, l-k},
\end{equation} which is the classical recurrence for the $q$-binomial coefficients \begin{equation}
 \label{Binom_rec} {m+n \brack m}_q = {m+n-1 \brack m}_q + q^{n} {m+n-1 \brack m-1}_q.
\end{equation} The correction term is necessary due to the fact that \eqref{Binom_rec} fails when both arguments $m$ and $n$ of the $q$-binomial coefficient are 0. By summing over both sides of \eqref{UG_F_recurrence} from $j=0$ to infinity we get a recurrence for ${}_l\S(N,q)$ as \begin{equation} \label{UG_L_recurrence} {}_l\S(N,q) = {}_l\S(N-1,q) + xq^N {}_l\S(N-k,q) + \delta_{N,l-k}.
\end{equation} This recurrence and the initial conditions  \begin{equation}\label{UG_L_initial_conds}
{}_l\S(i,q) = 0, \text{ if   }  i < l-k,
\end{equation} uniquely define the entire sequence for all integers $N$. We iterate the recurrence for the last term of \eqref{UG_L_recurrence}. This yields \begin{equation}\label{UG_L_iterated recurrence}
{}_l\S(N,q) = {}_l\S(N-1,q) + xq^N {}_l\S(N-k-1,q) + x^2q^{2N-k}{}_l\S(N-2k,q)+\delta_{N,l-k}+xq^{N}\delta_{N,l}.
\end{equation} This recurrence with the initial conditions \eqref{UG_L_initial_conds} defines ${}_l\S(N,q)$ uniquely for every integer $N$.

The recurrence for the right-hand side summand $\R_{m,n}(N,q)$ is given by 
\begin{align}
\nonumber\R_{m,n}(N,q) = &\R_{m,n}(N-1,q) + xq^N \R_{m-1,n}(N-k-1,q) + x^2q^{2N-k}\R_{m,n-2}(N-2k,q)\\\label{UG_G_recurrence} &+ \delta_{m,0}\,\delta_{n,0}\,\delta_{N,l-k} + xq^l\delta_{m,0}\,\delta_{n,1}\,\delta_{N,l}.
\end{align} The correction term is once again due to the recurrence of $q$-binomial coefficients \eqref{Binom_rec} failing when both the top and the bottom arguments are 0. One can easily prove this recurrence by employing the recurrence \eqref{Binom_rec} twice and simplifying terms. We apply the recurrence to the second $q$-binomial with the bottom argument $\lfloor n/2 \rfloor$, and we follow that up with applying the recurrence once again for $m$ to the $q$-binomial product with the bottom arguments $m$ and $n$. 

Summing over $m$ and $n$ from 0 to infinity one sees that ${}_r\S(N,q)$ satisfies the recurrence\begin{equation}\label{UG_R_recurrence}
{}_r\S(N,q) = {}_r\S(N-1,q) + xq^N {}_r\S(N-k-1,q) + x^2q^{2N-k}{}_r\S(N-2k,q)+\delta_{N,l-k}+xq^{N}\delta_{N,l}.
\end{equation} This recurrence with the initial terms \begin{equation}\label{UG_R_initial_conds}
{}_r\S(i,q) = 0, \text{ if   }  i < l-k,
\end{equation} uniquely defines the ${}_r\S(N,q)$ for all integers $N$.

Observe that the recurrences \eqref{UG_L_iterated recurrence} and \eqref{UG_R_recurrence} and their respective initial conditions \eqref{UG_L_initial_conds} and \eqref{UG_R_initial_conds} are identical. This finishes the proof of \eqref{UG_polynomial_equality}.
\end{proof}

\section{G\"ollnitz--Gordon partitions related double sum generating functions and their finitizations}\label{Sec_GG}

We start with the celebrated partition theorem of G\"ollnitz--Gordon.

\begin{theorem}[G\"ollnitz--Gordon Partition Theorems]\label{GG_Thm} For $i=1,2$, the number of partitions of $n$ into parts $\geq 2i-1$ with gaps between parts $\geq 2$ and no consecutive even parts, is the same as the number of partitions into parts congruent to $2+(-1)^i,4,6-(-1)^i$ modulo 8.
\end{theorem}

G\"ollnitz and Gordon independently discovered Theorem~\ref{GG_Thm} through the identities of Slater \cite{Slater_List}, where the identities 
\begin{align}
\label{GG1_orig_EQN}\sum_{n\geq0} \frac{q^{n^2+(1+(-1)^i)n}(-q;q^2)_n}{(q^2;q^2)_n} &= \frac{1}{(q^{2+(-1)^i},q^4,q^{6-(-1)^i};q^8)_\infty},
\end{align} for $i=1$ and 2 are presented. Let $\GG_i$, for $i=1,2$, denote the sets of partitions of Theorems~\ref{GG_Thm} with the gap conditions.  The left-hand side of \eqref{GG1_orig_EQN} is the generating functions for the sets $\GG_1$ and $\GG_2$, respective to $i=1$ and 2. More precisely, one can include the number of parts as the exponent of $x$ and write
\begin{align*}
\sum_{\pi\in\GG_i} x^{\# (\pi)} q^{|\pi|} = \sum_{n\geq0} \frac{x^n q^{n^2+(1+(-1)^i)n}(-q;q^2)_n}{(q^2;q^2)_n} 
\end{align*} for $i=1,2$.

In the recent years, Andrews--Bringmann--Mahlburg \cite{Andrews_Mahlburg_Bringmann} and Kur\c{s}ung\"oz \cite{Kagan1} have given new double sum representation for the generating function of partitions from the sets $\GG_i$. We would like to write down Kur\c{s}ung\"oz's theorem here.

\begin{theorem}[Kur\c{s}ung\"oz \cite{Kagan1}]\label{Kagan_GG_THM} Let 
\begin{align}
\label{GG_Kagan_size}K(m,n):&= \frac{3m^2-m}{2} + 4mn +4n^2,\\
\intertext{then we have} 
\label{GG1_Kagan_EQN} \sum_{\pi\in\GG_1} x^{\# (\pi)} q^{|\pi|} &= \sum_{m,n\geq 0} \frac{x^{m+2n}q^{K(m,n)}}{(q;q)_m(q^4;q^4)_n},\\
\label{GG2_Kagan_EQN} \sum_{\pi\in\GG_2} x^{\# (\pi)} q^{|\pi|} &= \sum_{m,n\geq 0} \frac{x^{m+2n}q^{K(m,n)+ 2m + 4n}}{(q;q)_m(q^4;q^4)_n}.
\end{align}
\end{theorem}

It is clear that the double sums on the right-hand sides of \eqref{GG1_Kagan_EQN} and \eqref{GG2_Kagan_EQN} are made out of objects with manifestly non-negative power series coefficients. Therefore, this representation is different than the ones of Andrews et.al. \cite{Andrews_Mahlburg_Bringmann}, where sign alterations are present in the sums. 

Kur\c{s}ung\"oz constructed these generating functions similar to the generating functions in the Section~\ref{Sec_Uniform_gap}. For completion, we will also construct generating functions for the sets $\GG_i$ briefly here before putting a restriction on the size of parts. The construction will follow the same steps as in Section~\ref{Sec_Uniform_gap}. The minimal configurations used in this construction is slightly different than Kur\c{s}ung\"oz's approach. This difference leads to different generating functions and hence new double sum identities.

For partitions in $\GG_1$, we can start with minimal configurations \begin{equation}\label{GG1_min_config}
{}_{\GG_1}\pi_{m,n} :=(\underbrace{1,3,5,\dots,2n-1}, 2n+2 ,2n+5, 2n+8\dots,2n-1 + 3m ),
\end{equation} where we have an initial chain of $n$ consecutive odd numbers followed up with $m$ singletons that have a gap of 3 with each adjacent part. Observe that the size of the minimal configuration ${}_{\GG_1}\pi_{m,n}$ is \[G (m,n):= \frac{3m^2+m}{2}+ 2mn+n^2\] as defined before in \eqref{GG1_min_config_size}. Notice that the size of these minimal configurations $G(m,n)$ is different from the $q$ factor, $K(m,n)$, used in Theorem~\ref{Kagan_GG_THM}.

The forward motion of the singletons are the same as in Section~\ref{Sec_Uniform_gap}. The splitting of two elements from the initial chain and the forward motions of the created pair afterwards can be defined as follows. Since partitions from $\GG_1$ do not allow consecutive evens, one needs to move a pair of consecutive odds forwards to the next pair of consecutive odds: \begin{equation}\label{GG_pair_motion}\underline{2k+1,2k+3} \mapsto \underline{2k+3,2k+5},\end{equation} for any $k\in\Z_{\geq 0}$. This forwards motion adds $4$ to the overall size of the partition. Moreover, the splitting pairs from the initial chain of ${}_{\GG_1}\pi_{m,n}$ is analogous to the case of \eqref{initial_chain_split}. We split pairs of consecutive odd integers from the initial chain as follows: \begin{equation}
\label{GG_pair_split} \underbrace{1,3,5,\dots,2n-1}\mapsto \underbrace{1,3,5,\dots,2n-5},\underline{2n-1,2n+1}.
\end{equation} This motion also adds 4 to the overall size of the partition. 

The motions \eqref{GG_pair_motion} and \eqref{GG_pair_split} are defined assuming that the outcomes of these motions do not violate the difference conditions of G\"ollnitz--Gordon Partition Theorems. In general a singleton might be close to a pair and the free forwards motion might not be possible. For these cases, once again, we need to address the pairs crossing over the singletons in the forward motion. We will do it in two cases; when the singleton has an even value and when it has an odd value. Before any crossover, the minimal gap between the greater value of a pair and the even singleton is 3. For some non-negative integer $k$, the crossing of a pair over an even singleton is defined by \begin{align}
\label{GG_pair_even_crossover} \underline{2k+1,2k+3}, 2k+6 &\mapsto 2k+2,\underline{2k+5,2k+7}.
\intertext{Similarly, for the odd singleton case the minimal gap between the greater value of a pair and the singleton is 2, and the crossing over can be defined as}
\label{GG_pair_odd_crossover} \underline{2k+1,2k+3}, 2k+5 &\mapsto 2k+1,\underline{2k+5,2k+7}.
\end{align} It is clear that these forward motions are reversible. Moreover, we would like to point out that, just as in the case of the all the forward motion of the pairs, in the crossovers the size of the partition raises by 4. 

To emphasize that these motions are reversible, we would like to give an example. Let \[\pi = (2,5,7,9,13,17,19,22) \in \GG_1.\] We first check to see if there is an initial chain (consecutive odd numbers starting from 1). In this case, initial chain is not present. Then we move on to the identification of the pairs. Starting from the smallest part (that is larger than the end of the initial chain) we group each two consecutive odd parts other and underline these pairs without using a part in more than one pair: \[\pi = (2,\underline{5,7},9,13,\underline{17,19},22).\] This is enough to identify the minimal configuration this partition descends from. The total number of elements in the initial chain and the pairs is $n$ and the number of underlined elements is $m$. In particular, in this example $m=n=4$, and $\pi$ is a descendant of the minimal configuration \[{}_{\GG_1}\pi_{4,4} = (\underbrace{1,3,5,7},10,13,16,19).\] To revert $\pi$ back to ${}_{\GG_1}\pi_{4,4}$, we need to first need to move the pairs back to their original positions. We do this starting from the smallest pair by using the motion rules \eqref{GG_pair_motion}, \eqref{GG_pair_split}, \eqref{GG_pair_even_crossover}, and \eqref{GG_pair_odd_crossover} in reverse. For the pair $\underline{5,7}$ of $\pi$, we just need to use \eqref{GG_pair_even_crossover} in reverse once. After this motion, this pair cannot move any further back and it becomes the initial chain: \[(\underbrace{1,3},6,9,13,\underline{17,19},22).\] After this, we need to apply \eqref{GG_pair_odd_crossover} in reverse twice followed by \eqref{GG_pair_even_crossover} in reverse to get \[(\underbrace{1,3,5,7},10,13,17,22).\] Then, we can reach ${}_{\GG_1}\pi_{4,4}$ by moving the singletons back to their original positions.

Finally, given a minimal configuration ${}_{\GG_1}\pi_{m,n}$ only $\lfloor n/2 \rfloor$ pairs can be split from the initial chain of $n$ consecutive odd integers, and the related generating function for the motion of these pairs is $(q^4;q^4)^{-1}_{\lfloor n/2\rfloor},$ since any forward motion of these pairs adds $4$ to the total size of the partition ${}_{\GG_1}\pi_{m,n}$. Once again note that all these motions are bijective and any given partition from $\GG_1$ can be traced back to its minimal configuration by using the motions in the reverse order. Hence, we get \begin{equation*}\sum_{\pi\in\GG_1}x^{\#(\pi)}q^{|\pi|} = \sum_{m,n\geq 0} \frac{x^{m+n}q^{G(m,n)}}{(q;q)_m(q^4;q^4)_{\lfloor n/2\rfloor}}.\end{equation*}

Similarly, we can start with the minimal configuration \begin{align}
\label{GG2_min_config1} {}_{\GG_2}\pi_{m,n} & :=(\underbrace{3,5,\dots,2n+1}, 2n+4 ,2n+7, \dots,2n+1 + 3m ).
\end{align} We use the same forward motion rules of $\GG_1$ case. The sizes of the minimal configuration \eqref{GG2_min_config1} is $G(m,n)+2m+2n.$ Hence, we have the following theorem.

\begin{theorem}\label{GG_Ali_THM}We have \begin{align} 
\label{GG1_Ali_EQN}\sum_{\pi\in\GG_1}x^{\#(\pi)}q^{|\pi|} &= \sum_{m,n\geq 0} \frac{x^{m+n}q^{G(m,n)}}{(q;q)_m(q^4;q^4)_{\lfloor n/2\rfloor}},\\
\label{GG2_Ali_EQN}\sum_{\pi\in\GG_2}x^{\#(\pi)}q^{|\pi|} &= \sum_{m,n\geq 0} \frac{x^{m+n}q^{G(m,n)+2m+2n}}{(q;q)_m(q^4;q^4)_{\lfloor n/2\rfloor}}, 
\end{align} where $G(m,n)$ is defined as in \eqref{GG1_min_config_size}.
\end{theorem}

Note that \eqref{GG_Kagan_size} with \eqref{GG1_min_config_size} are different from each other. Although the $q$-Pochhammer bases are the same, the generating function representations \eqref{GG1_Kagan_EQN} and \eqref{GG1_Ali_EQN}, and \eqref{GG1_Kagan_EQN} and \eqref{GG2_Ali_EQN} are distinct from each other, respectively. This leads to the following series transformation formulas after we do a even-odd split for the variable $n$ in \eqref{GG_Ali_THM}.

\begin{theorem}\label{GG_Ali_Kagan_Together}
\begin{align}
\sum_{m,n\geq 0} \frac{x^{m+2n}q^{\frac{3m^2-m}{2}+4mn+4n^2}}{(q;q)_m(q^4;q^4)_n} &= \sum_{m,n\geq 0} \frac{x^{m+2n}q^{\frac{3m^2-m}{2}+4mn+4n^2+m} (1+xq^{2m+4n+1})}{(q;q)_m(q^4;q^4)_n},\\
\sum_{m,n\geq 0} \frac{x^{m+2n}q^{\frac{3m^2-m}{2}+4mn+4n^2+2m+4n} }{(q;q)_m(q^4;q^4)_n} &= \sum_{m,n\geq 0} \frac{x^{m+2n}q^{\frac{3m^2-m}{2}+4mn+4n^2+3m+4n} (1+xq^{2m+4n+3})}{(q;q)_m(q^4;q^4)_n}.
\end{align}
\end{theorem}

One can also prove these identities relying on the exponent of the variable $x$.  Fixing the exponent of $x$ to some $k$ one can write \[\sum_{k\geq 0} x^k q^{\frac{3k^2-k}{2}}  \sum_{n\geq 0} \alpha_{k,n}(q) = \sum_{k\geq 0} x^k q^{\frac{3k^2-k}{2}} \sum_{n\geq 0} \left(\beta_{k,n}(q) +  \gamma_{k,n}(q)\right),\] for some sequences $\alpha_{k,n}(q)$, $\beta_{k,n}(q)$, and $\gamma_{k,n}(q)$. Then in these cases one can easily show that \[\alpha_{k,n}(q) = \beta_{k,n}(q)+\gamma_{k,n}(q)\] using only elementary cancellations after carrying $\beta_{k,n}(q)$ or $\gamma_{k,n}(q)$ on the other side of the claimed identity.

Moreover, both Theorem~\ref{Kagan_GG_THM} and Theorem~\ref{GG_Ali_THM} are found using a constructive combinatorial structure. This structure can be utilized to find finite analogs of these generating functions. It is clear that putting a bound on the exponent of $x$ in either theorem would bound the number of parts of the partitions. Our plan is to put a bound on the largest part of partitions in the spirit of \cite{BerkovichUncu7}. Since the minimal configuration and the forward motions are explicitly stated here, we would like to start with the finitization of the double sums of Theorem~\ref{GG_Ali_THM}. 

Let $N$ be a non-negative integer. For the first G\"ollnitz--Gordon type partitions we start by looking at the largest singleton of the minimal configuration ${}_{\GG_1}\pi_{m,n}$, $3m+2n-1$. This part can only move $N-3m-2n+1$ steps forward. Analogous to the finitization in Section~\ref{Sec_Uniform_gap} this shows that the related generating function for this bounded motion of $m$ singletons is \begin{equation}\label{GG1_singleton_Bounded_move}
{N-3m-2n+1 +m \brack m}_q.
\end{equation} The pairs of ${}_{\GG_1}\pi_{m,n}$ move in steps of two when crossing a singleton is out of the picture. Hence, without the consideration of the singletons the largest pair $\underline{2n-3,2n-1}$ to split from the initial chain of ${}_{\GG_1}\pi_{m,n}$ can move $\lfloor(N-(2n-1))/2\rfloor$ times before it reaches the boundary. The motions \eqref{GG_pair_even_crossover} and \eqref{GG_pair_odd_crossover} shows that a pair darts forward 4 steps instead of 2 when it crosses over a singleton. Hence, a pair loses one possible move every time it crosses over one of the $m$ singletons. This way we see that the largest pair to be $\underline{2n-3,2n-1}$ can move a total of $\lfloor(N-(2n-1))/2\rfloor-m$ times. Also recall that any motion of the pairs add 4 to the total size of the partitions. Hence the motion of the pairs with the bound of the largest part $N$ is \begin{equation}\label{GG1_pair_Bounded_move}
{\lfloor\frac{N-(2n-1)}{2}\rfloor-m+\lfloor \frac{n}{2}\rfloor \brack \lfloor\frac{n}{2}\rfloor}_{q^4}.
\end{equation}

Similarly for the second G\"ollnitz--Gordon partitions related generating function we get the related $q$-binomial coefficients
\begin{equation}\label{GG2_Bounded_moves}
{N-3m-2n-1 +m \brack m}_q\text{ and }{\lfloor\frac{N-(2n+1)}{2}\rfloor-m+\lfloor \frac{n}{2}\rfloor \brack \lfloor\frac{n}{2}\rfloor}_{q^4}.
\end{equation} 

Replacing the $q$-Pochhammer symbols with the related $q$-binomials \eqref{GG1_singleton_Bounded_move}, \eqref{GG1_pair_Bounded_move}, and \eqref{GG2_Bounded_moves} after making slight simplifications in the floor functions in the related identities we get the following theorem.

\begin{theorem}\label{GG_Ali_BDD_THM}Let $G(m,n)$ be as in \eqref{GG1_min_config_size} and $N$ be a non-negative integer. We have  \begin{align} 
\label{GG1_Ali_BDD_EQN}\sum_{\pi\in\GG_{1,N}}x^{\#(\pi)}q^{|\pi|} &= \sum_{m,n\geq 0}x^{m+n}q^{G(m,n)}{N-2m-2n+1 \brack m}_q{\lfloor\frac{N+1}{2}\rfloor-m-n+\lfloor \frac{n}{2}\rfloor \brack \lfloor\frac{n}{2}\rfloor}_{q^4},\\
\label{GG2_Ali_BDD_EQN}\sum_{\pi\in\GG_{2,N}}x^{\#(\pi)}q^{|\pi|} &= \sum_{m,n\geq 0} x^{m+n}q^{G(m,n)+2m+2n}{N-2m-2n-1  \brack m}_q{\lfloor\frac{N-1}{2}\rfloor-m-n+\lfloor \frac{n}{2}\rfloor \brack \lfloor\frac{n}{2}\rfloor}_{q^4} ,
\end{align} where $\GG_{i,N}$ is the set of partitions from $\GG_i$ with the extra bound $N$ on the largest part of the partitions. 
\end{theorem}

We use the same techniques on the minimal configurations and the forward motions to get the finite version of Kur\c{s}ung\"oz's Theorem (Theorem~\ref{Kagan_GG_THM}).

\begin{theorem}\label{GG_Kagan_BDD_THM}Let $K(m,n)$ be as in \eqref{GG_Kagan_size} and $N$ be a non-negative integer. We have  \begin{align} 
\label{GG1_Kagan_BDD_EQN}\sum_{\pi\in\GG_{1,N}}x^{\#(\pi)}q^{|\pi|} &= \sum_{m,n\geq 0}x^{m+2n}q^{K(m,n)}{N-2m-4n+2 \brack m}_q{\lfloor\frac{N+1}{2}\rfloor-m-n \brack n}_{q^4},\\
\label{GG2_Kagan_BDD_EQN}\sum_{\pi\in\GG_{2,N}}x^{\#(\pi)}q^{|\pi|} &= \sum_{m,n\geq 0} x^{m+2n}q^{K(m,n)+2m+4n}{N-2m-4n  \brack m}_q{\lfloor\frac{N-1}{2}\rfloor-m-n \brack n}_{q^4} ,
\end{align} where $\GG_{i,N}$ is the set of partitions from $\GG_i$ with the extra bound $N$ on the largest part of the partitions. 
\end{theorem}

The right-hand sides of \eqref{GG1_Ali_BDD_EQN} and \eqref{GG1_Kagan_BDD_EQN} (also \eqref{GG2_Ali_BDD_EQN} and \eqref{GG2_Kagan_BDD_EQN}) are proven to be the same by the combinatorial construction. One can directly prove these equalities using the exponent of $x$ (analogous to the explanation after Theorem~\ref{GG_Ali_Kagan_Together}) and recurrence relation for the $q$-binomial coefficients \eqref{Binom_rec}. Another direct proof can be done by using recurrences similar to Section~\ref{Sec_Uniform_gap}. This standard $q$-series proof is easily followed with the aid of computer implementations Sigma \cite{Sigma} and/or qMultiSum \cite{qMultiSum}. We will use these symbolic computation implementations' help in the proof of Theorem~\ref{Fin_GG_Ali_THM} later in this section.

More importantly, the polynomials presented in Theorem~\ref{GG_Ali_BDD_THM} and \ref{GG_Kagan_BDD_THM} are not the first finitizations for the generating function of the number of partitions with the G\"ollnitz--Gordon gap conditions.  Andrews in \cite{Andrews_q_Tri} finds explicit formulas for the generating function for the number of partitions from the sets $\GG_{i,2N+1}$, for $i=1,2$ using $q$-trinomial coefficients. 

We would like to note that, although we will not be focusing on the $q$-trinomial coefficients in the work presented here, they are crucial for the polynomial identities coming from partition theoretic background. Some resources for the interested reader are as follows. The $q$-trinomial coefficients have been studied extensively by Andrews and Baxter \cite{Andrews_Baxter}. Some refinements of the $q$-trinomial coefficients can be found in Warnaar's work \cite{Warnaar_Refined, Warnaar_T}. Some new results on $q$-trinomial coefficients in connection with the partition theory (especially Capparelli's partition theorem) can be found in the author's work joint with Berkovich \cite{BerkovichUncu7, BerkovichUncu8, BerkovichUncu9}. 

Another finitization of the related generating functions with gap conditions is due to Berkovich--McCoy--Orrick. 
\begin{theorem}[Berkovich, McCoy, Orrick \cite{Berkovich_McCoy_Orrick}]\label{GG_Berkovich_McCoy_Orrick_THM}For any non-negative integer $N$, we have 
\begin{align}
\label{GG1_Berkovich_McCoy_Orrick_oddBDD}\sum_{\pi\in\GG_{1,2N+1}}q^{|\pi|} &= \sum_{m,n\geq 0} q^{m^2+n^2} {N-m+1\brack n}_{q^2}{n\brack m}_{q^2}\\
\label{GG2_Berkovich_McCoy_Orrick_oddBDD}\sum_{\pi\in\GG_{2,2N+1}}q^{|\pi|} &= \sum_{m,n\geq 0} q^{m^2+n^2+ 2n} {N-m\brack n}_{q^2}{n\brack m}_{q^2}.
\end{align}
\end{theorem}

The $q$-binomials in the generating functions in Theorem~\ref{GG_Ali_BDD_THM} (or \eqref{GG_Kagan_BDD_THM}) and \ref{GG_Berkovich_McCoy_Orrick_THM} have different bases. The equality of \eqref{GG1_Ali_BDD_EQN}  with $x=1$, $N\mapsto 2N+1$ and \eqref{GG1_Berkovich_McCoy_Orrick_oddBDD}, and \eqref{GG2_Ali_BDD_EQN} with $x=1$, $N\mapsto 2N+1$ and \eqref{GG2_Berkovich_McCoy_Orrick_oddBDD}, respectively is Theorem~\ref{Fin_GG_Ali_THM}. Note that Theorem~\ref{Fin_GG_Ali_THM} is already proven by the generating function interpretations; regardless we would like to give a direct proof of this identity using recurrences. 

\begin{proof}[Proof of Theorem~\ref{Fin_GG_Ali_THM}]
For $i=1,2$, let $\R_{i,m,n}(N,q)$ be the right-hand side summands of \eqref{Fin_GG1_Intro_EQN} and \eqref{Fin_GG2_Intro_EQN}, respectively. Both Sigma \cite{Sigma} and/or qMultiSum \cite{qMultiSum} packages can find and automatically prove that these functions satisfy the recurrence
\begin{align}
\label{Berkovich_McCoy_Summand_Rec} \R_{i,m,n}(N+2,q) &= \R_{i,m,n}(N+1,q)+ q^{2N+5}\R_{i,m,n-1}(N+1,q) + q^{2N+4} \R_{i,m-1,n-1}(N,q). 
\intertext{Let ${}_r\S_i(N,q)$ be the sum on the right-hand side of \eqref{Fin_GG1_Intro_EQN} and \eqref{Fin_GG2_Intro_EQN} for $i =1,2$, respectively. By summing \eqref{Berkovich_McCoy_Summand_Rec} over $m,n\geq 0$ we see that ${}_r\S_i(N,q)$ satisfies }
\label{Berkovich_McCoy_Rec}{}_r\S_i(N+2,q) &= (1+ q^{2N+5})\,{}_r\S_{i}(N+1,q) + q^{2N+4}\, {}_r\S_{i}(N,q).
\intertext{We can iterate this recurrence by writing this recurrence with $N\mapsto N+1$ and then using \eqref{Berkovich_McCoy_Rec} directly for the term $q^{2N+7}{}_r\S_i(N+2,q)$. This yields the recurrence}
\label{Berkovich_McCoy_Iterated_Rec}  {}_r\S_i(N+3,q) &=  {}_r\S_i(N+2,q) + q^{2N+6}(1+q+q^{2N+6}) {}_r\S_i(N+1,q) + q^{4N+11} {}_r\S_i(N,q)
\intertext{Similarly we would like to find the recurrences for the left-hand sides of \eqref{Fin_GG1_Intro_EQN} and \eqref{Fin_GG2_Intro_EQN}. In this case we need to split the summation variable $n$ with respect to its parity with the use of symbolic summation tools. For $i=1,2$, let $\L_{i,m,n}(N)$ be the left-hand side summands of  \eqref{Fin_GG1_Intro_EQN} and \eqref{Fin_GG2_Intro_EQN}, respectively. Then once again Sigma \cite{Sigma} and/or qMultiSum \cite{qMultiSum} packages prove that these functions satisfy the recurrence}
\nonumber \L_{i,m,2n+\nu}(N+3) &= \L_{i,m,2n+\nu}(N+2)+q^{2N+6}(1+q)\L_{i,m-1,2n+\nu}(N+1)
\\\label{Ali_GG_left_Summand_Rec}&+q^{4N+12}\L_{i,m,2(n-1)+\nu}(N+1)+q^{4N+11}\L_{i,m-2,2n+\nu}(N),
\intertext{for $\nu = 0,1$. Then summing both sides of \eqref{Ali_GG_left_Summand_Rec} over $m,n\geq 0$, we get}
\label{Ali_GG_Rec} {}_l\S_i(N+3,q) &= {}_l\S_i(N+2,q) +  q^{2 N+6} (1 + q + q^{2 N+6}){}_l\S_i(N+1,q) + q^{4N+11} {}_l\S_i(N,q),
\end{align}
where ${}_l\S_i(N,q)$ represent the left-hand side summand of \eqref{Fin_GG1_Intro_EQN} and \eqref{Fin_GG2_Intro_EQN} for $i=1,2$, respectively. Notice that \eqref{Berkovich_McCoy_Iterated_Rec} and \eqref{Ali_GG_Rec} are the same order 3 recurrence. Therefore checking the first three initial values for both sides is enough to finish the proof. The initial values
\[\begin{array}{ll}
 {}_l\S_1(-1,q) = {}_r\S_1(-1,q) =1, 			& {}_l\S_2(-1,q) = {}_r\S_2(-1,q) =0, \\
 {}_l\S_1(0,q) = {}_r\S_1(0,q) =1+q, 			& {}_l\S_2(0,q) = {}_r\S_2(0,q) =1,\\
 {}_l\S_1(1,q) = {}_r\S_1(1,q) =1+q+q^2+q^3+q^4, & {}_l\S_2(1,q) = {}_r\S_2(1,q) =1+q^3 
\end{array}\] with the recurrences \eqref{Berkovich_McCoy_Iterated_Rec} and \eqref{Ali_GG_Rec} show that the identities \eqref{Fin_GG1_Intro_EQN} and \eqref{Fin_GG2_Intro_EQN} hold for any non-negative integer $N$.
\end{proof}

One corollary of Theorem~\ref{Fin_GG_Ali_THM} is the double sum equalities we get as we tend $N$ to infinity.

\begin{corollary}\label{Corollary_GG_Ali_Berkovich} We have
\begin{align*}
\sum_{m,n\geq 0} \frac{q^{\frac{3m^2+m}{2}+ 2mn+n^2}}{(q;q)_m(q^4;q^4)_{\lfloor n/2\rfloor}} &= \sum_{m,n\geq 0 } \frac{q^{2m^2+2mn+n^2}}{(q^2;q^2)_m(q^2;q^2)_n},\\
\sum_{m,n\geq 0} \frac{q^{\frac{3m^2+m}{2}+ 2mn+n^2+2m+2n}}{(q;q)_m(q^4;q^4)_{\lfloor n/2\rfloor}} &=\sum_{m,n\geq 0 } \frac{q^{2m^2+2mn+n^2+2m+2n}}{(q^2;q^2)_m(q^2;q^2)_n}. 
\end{align*}
\end{corollary}

\begin{proof}On the right-hand side of the identities of Theorem~\ref{Fin_GG_Ali_THM}, we do the substitution $k=n-m$, use \begin{equation}
 \label{Binom_limit} \displaystyle \lim_{N\rightarrow\infty}{N\brack m}_q = \frac{1}{(q;q)_m}
\end{equation} and the property \begin{equation}
 \label{1over_neg_Pochhammer} \frac{1}{(q;q)_n} = 0, \text{ if }n<0.
\end{equation} We then take the limit $N\rightarrow\infty$ followed by $k\mapsto n$.\end{proof}

\section{Little G\"ollnitz partitions related double sum generating functions and their finitizations}\label{Sec_lG}

Another partition theorem pair that is structurally similar to G\"ollnitz--Gordon theorems is the little G\"ollnitz theorems.

\begin{theorem}[Little G\"ollnitz Partition Theorems]\label{lG_Thm} For $i=1,2$, the number of partitions of $n$ into parts $\geq i$ with gaps between parts $\geq 2$ and no consecutive odd parts, is the same as the number of partitions into parts congruent to $i,4-(-1)^i,5+i$ modulo 8.
\end{theorem}

These partition theorems have analytic forms similar to G\"ollnitz--Gordon identities. These identities are
\begin{align}
\label{lG1_orig_EQN}\sum_{n\geq0} \frac{q^{n^2+n}(-1/q;q^2)_n}{(q^2;q^2)_n} &= \frac{1}{(q,q^5,q^6;q^8)_\infty},\\
\label{lG2_orig_EQN}\sum_{n\geq0} \frac{q^{n^2+n}(-q;q^2)_n}{(q^2;q^2)_n} &= \frac{1}{(q^2,q^3,q^7;q^8)_\infty},
\end{align} for $i=1$ and $2$, respectively.

We can write a double sum representation of the generating function for the little G\"ollnitz partitions with the gap conditions.

\begin{theorem}\label{lG_double_sum_Thm} Let 
\begin{align}
\label{lG_size}L(m,n)&:= 3\frac{m^2+m}{2} + 2nm + n^2 + n
\intertext{then we have}
\label{lG1_Ali_EQN}\sum_{\pi\in\lG_1} x^{\#(\pi)}q^{|\pi|} &= \sum_{m,n\geq 0} \frac{x^{m+n}q^{L(m,n)}}{(q;q)_m(q^4;q^4)_{\lfloor n/2 \rfloor}} +  \sum_{m,n\geq 0} \frac{x^{m+n+1}q^{L(m,n)+2m+2n+1}}{(q;q)_m(q^4;q^4)_{\lfloor n/2 \rfloor}} ,\\
\label{lG2_Ali_EQN}\sum_{\pi\in\lG_2} x^{\#(\pi)}q^{|\pi|} &= \sum_{m,n\geq 0} \frac{x^{m+n}q^{L(m,n)}}{(q;q)_m(q^4;q^4)_{\lfloor n/2 \rfloor}},
\end{align} where $\lG_i$ is the set of partitions that satisfy the gap conditions of Theorem~\ref{lG_Thm} with the respective $i$.
\end{theorem}

The double sum representations of these generating functions are made out of terms that have manifestly non-negative power series coefficients. This is once again different than the double series representations of Andrews--Bringmann-Mahlburg \cite{Andrews_Mahlburg_Bringmann}.

Note that the difference of \eqref{lG1_Ali_EQN} and \eqref{lG2_Ali_EQN} with $x=1$, and the little G\"ollnitz products (the right-hand sides of \eqref{lG1_orig_EQN} and \eqref{lG2_orig_EQN}) yield Theorem~\ref{Intro_THM_little_G}.

To prove \eqref{lG_double_sum_Thm}, we need to give the minimal constructions and the motions of the parts of partitions as in Section~\ref{Sec_GG}. The two possible minimal configurations where there are $n$ initially placed close parts followed by $m$ sinlgletons are
\begin{align}
\label{lG_min_config1} {}_{1}\pi_{m,n} &:=(\underbrace{2,4,6,\dots,2n}, 2n+3 ,2n+6, 2n+9\dots,2n + 3m ), \\
\label{lG_min_config2} {}_{2}\pi_{m,n} &:=(1,\underbrace{4,6,8,\dots,2n+2} ,2n+5, 2n+8\dots,2n+2+ 3m ), 
\end{align} where in \eqref{lG_min_config2} the part 1 is the sole singleton that does not move. Note that the sizes of ${}_{1}\pi_{m,n}$ and ${}_{2}\pi_{m,n}$ are $L(m,n)$ as defined in \eqref{lG_size} and $L(m,n)+2n+2m+1$, respectively. Also the number of parts of ${}_{1}\pi_{m,n}$ and ${}_{2}\pi_{m,n}$ are $m+n$ and $m+n+1$, respectively.

Motions of the singletons are defined as before. When it comes to splitting from the initial chains and moving forward we need to move a pair of consecutive evens to the next consecutive evens (analogous to the odd pair case of G\"ollnitz--Gordon identities) and with this free forwards motion \begin{equation}\label{lG_pair_move}
\underline{2k, 2k+2}\mapsto \underline{2k+2, 2k+4},
\end{equation} we see that we need to add 4 at every move. The crossing of consecutive even pairs over singletons can be defined by the following motions. For any positive integer $k$ the forward motion of an even pair over singletons is
\begin{align}
\label{lG_crossing_over_even_singleton} \underline{2k, 2k+2}, 2k+4 &\mapsto 2k,\underline{2k+4,2k+6},\intertext{or}
\label{lG_crossing_over_odd_singleton} \underline{2k, 2k+2}, 2k+5 &\mapsto 2k+1,\underline{2k+4, 2k+6}.
\end{align}
These motions are reversible. With the inverses of these motions any partition from $\lG_1$ can be traced back to its minimal configuration. If the partition has the smallest part $1$ then it is an image of ${}_{2}\pi_{m,n}$ for some non-negative integers $m$ and $n$. If 1 is not a part then that partition also lies in the set $\lG_2$ and can be traced back to ${}_{1}\pi_{m,n}$ for some non-negative integers $m$ and $n$. This finishes the proof of Theorem~\ref{lG_double_sum_Thm}.

As we included detailed examples for the backwards tracing of partitions in Sections~\ref{Sec_Uniform_gap} and \ref{Sec_GG}, we omit it here.

Similar to Sections~\ref{Sec_Uniform_gap} and \ref{Sec_GG}, we can easily impose a bound in the largest part of partitions counted. This yields refinements of the generating functions for the number of partitions from sets $\lG_i$ for $i=1,2$. For a non-negative integer $N$ the largest singletons of ${}_{1}\pi_{m,n}$ and ${}_{2}\pi_{m,n}$ ( $3m+2n$ and $3m+2n+2$, respectively) can move forward $N-3m-2n$ and $N-3m-2n-2$ steps, respectively. Therefore, the respective generating functions for the forwards motions of the $m$ singletons are \begin{equation}\label{lG_BDD_singleton_moves}
{N-3m-2n + m \brack m}_q,\text{ and }{N-3m-2n-2 + m \brack m}_q.
\end{equation} 
Similar to Theorem~\ref{Unif_Gap_THM_Bounded} the second $q$-binomial in \eqref{lG_BDD_singleton_moves} misses to count the single partition $(1)$ when the $N=1$. This is due to the assumption that the largest singleton (that can move) is $3m+2n+2 > 1$. This issue does not arise for $N>1$. This error can easily corrected with a Kronecker delta function in our calculations. 

The forward motions of the $\lfloor n/2\rfloor$ consecutive even pairs that can split from the initial chain are \[\left\lfloor \frac{N-2n}{2}\right\rfloor-m, \text{ and }\left\lfloor \frac{N-(2n+2)}{2}\right\rfloor-m,\] where, analogous to the G\"ollnitz--Gordon case, we have a floor function due to every forward motion carrying the pairs 2 steps instead of 1, and the subtraction of $m$ is due to the possible crossing over the $m$ singletons losing the single extra motion. Therefore, the respective generating functions for the forward motions of the $\lfloor n/2 \rfloor$ pairs are \begin{equation}\displaystyle
{\left\lfloor \frac{N-2n}{2}\right\rfloor-m+\left\lfloor \frac{n}{2}\right\rfloor \brack \left\lfloor \frac{n}{2}\right\rfloor}_{q^4} \text{ and }{\left\lfloor \frac{N-2n-2}{2}\right\rfloor-m+\left\lfloor \frac{n}{2}\right\rfloor \brack \left\lfloor \frac{n}{2}\right\rfloor}_{q^4}. 
\end{equation}
This construction yields the following theorem after simplifying the floor functions.

\begin{theorem}\label{GG_Ali_BDD_THM}Let $L(m,n)$ be as in \eqref{lG_size} and $N$ be a non-negative integer. We have  \begin{align} 
\label{lG1_Ali_BDD_EQN}\sum_{\pi\in\lG_{1,N}}x^{\#(\pi)}q^{|\pi|} &= \sum_{m,n\geq 0}x^{m+n}q^{L(m,n)}{N-2m-2n \brack m}_q {\left\lfloor \frac{N}{2}\right\rfloor-m-n+\left\lfloor \frac{n}{2}\right\rfloor \brack \left\lfloor \frac{n}{2}\right\rfloor}_{q^4} \\\nonumber&\hspace{-1cm}+ \sum_{m,n\geq 0}x^{m+n+1}q^{L(m,n)+2m+2n+1}{N-2m-2n-2 \brack m}_q{\left\lfloor \frac{N}{2}\right\rfloor-m-n-1+\left\lfloor \frac{n}{2}\right\rfloor \brack \left\lfloor \frac{n}{2}\right\rfloor}_{q^4}+ \delta_{N,1}\,xq
,\\
\label{lG2_Ali_BDD_EQN}\sum_{\pi\in\lG_{2,N}}x^{\#(\pi)}q^{|\pi|} &= \sum_{m,n\geq 0} x^{m+n}q^{L(m,n)}{N-2m-2n \brack m}_q{\left\lfloor \frac{N}{2}\right\rfloor-m-n+\left\lfloor \frac{n}{2}\right\rfloor \brack \left\lfloor \frac{n}{2}\right\rfloor}_{q^4} ,
\end{align} where $\lG_{i,N}$ is the set of partitions from $\GG_i$ with the extra bound $N$ on the largest part of the partitions, and $\delta_{i,j}$ is the Kronecker delta.
\end{theorem}

Alladi and Berkovich \cite{Alladi_Berkovich} have formulated a double sum representation of the generating functions for the little G\"ollnitz partitions with gap conditions. This is similar to the Theorem~\ref{GG_Berkovich_McCoy_Orrick_THM} in structure.

\begin{theorem}[Alladi, Berkovich \cite{Alladi_Berkovich}]\label{lG_Alladi_Berkovich} We have
\begin{align}
\label{lG1_Alladi_Berkovich_EQN}\sum_{\pi\in\lG_{1}}q^{|\pi|} &= \sum_{m,n\geq 0} \frac{q^{m^2+2mn+2n^2+m-n}}{(q^2;q^2)_m(q^2;q^2)_n},\\
\label{lG2_Alladi_Berkovich_EQN}\sum_{\pi\in\lG_{2}}q^{|\pi|} &= \sum_{m,n\geq 0} \frac{q^{m^2+2mn+2n^2+m+n}}{(q^2;q^2)_m(q^2;q^2)_n}.
\end{align}
\end{theorem}

Notice that comparison of Theorems~\ref{lG_double_sum_Thm} and \ref{lG_Alladi_Berkovich} directly yields a similar result to Corollary~\ref{Corollary_GG_Ali_Berkovich}. On the other hand, the finite analog of Theorem~\ref{lG_Alladi_Berkovich} is not clear. Therefore, a similar result to Theorem~\ref{Fin_GG_Ali_THM} is harder to achieve. 

\begin{theorem} For any non-negative integer $N$ we have
\begin{align}
\label{lG_BDD_transform1} \sum_{m,n\geq 0} &q^{L(m,n)}{2N-2m-2n \brack m}_q{N-m-n+\left\lfloor \frac{n}{2}\right\rfloor \brack \left\lfloor \frac{n}{2}\right\rfloor}_{q^4}= \sum_{m,n\geq 0} q^{m^2+n^2+n} {N-m\brack n}_{q^2} {n\brack m}_{q^2}.\\
\nonumber\sum_{m,n\geq 0}&q^{L(m,n)+2m+2n+1}{2N-2m-2n-2 \brack m}_q{N-m-n-1+\left\lfloor \frac{n}{2}\right\rfloor \brack \left\lfloor \frac{n}{2}\right\rfloor}_{q^4}\\\label{lG_BDD_transform2} &\hspace{4cm}= \sum_{m,n\geq 0} q^{m^2+n^2+3n+1} {N-m-1\brack n}_{q^2} {n\brack m}_{q^2},
\end{align} where $L(m,n)$ is defined as in \eqref{lG_size}.
\end{theorem}

\begin{proof} Analogous to the proof of Theorem~\ref{Fin_GG_Ali_THM}, we will show that both sides of the identities satisfy the same recurrences. On the left-hand sides we need to split the even and odd cases of the summation variable $n$, but besides that it is standard and automatic to get the recurrences for both sides of these identities using the computer algebra implementations Sigma \cite{Sigma} and/or qMultiSum \cite{qMultiSum}. Once again, for $i=1,2$, let $\L_{i,m,n}(N,q)$ and $\R_{i,m,n}(N,q)$ be the left-hand and right hand side summands of \eqref{lG_BDD_transform1} and \eqref{lG_BDD_transform2}, respectively. We get that these functions satisfy the recurrences 
\begin{align}
\nonumber\L_{i,m,2n+\nu}(N+3,q) &= \L_{i,m,2n+\nu}(N+2,1)+q^{2N+5}(1+q)\L_{i,m-1,2n+\nu}(N+1,q)
\\\nonumber&\hspace{1cm}+q^{4N+10}\L_{i,m,2(n-1)+\nu}(N+1,q)+q^{4N+9}\L_{i,m-2,2n+\nu}(N,q),
\intertext{for $\nu=0,1$, and}
\nonumber\R_{i,m,n}(N+2,q) &= \R_{i,m,n}(N+1,q)+ q^{2N+4}\R_{i,m,n-1}(N+1,q) + q^{2N+3} \R_{i,m-1,n-1}(N,q). 
\intertext{This shows that the left-hand and right-hand side summations (denoted by ${}_l\S(N,q)$ and ${}_r\S(N,q)$, respectively) satisfy}
\label{lG_BDD_left_Rec}{}_l\S_i(N+3,q) &= {}_l\S_i(N+2,q) +  q^{2 N+5} (1 + q + q^{2 N+5}){}_l\S_i(N+1,q) + q^{4N+9} {}_l\S_i(N,q),
\intertext{and}
\label{Berkovich_McCoy_type_Rec_lG}{}_r\S_i(N+2,q) &= (1+ q^{2N+4})\,{}_r\S_{i}(N+1,q) + q^{2N+3}\, {}_r\S_{i}(N,q).
\end{align}
Iterating \eqref{Berkovich_McCoy_type_Rec_lG}, we see that ${}_r\S_i(N,q)$ also satisfy \eqref{lG_BDD_left_Rec}. Therefore showing the equality of the first 3 terms for both sides is enough to validate the identities of \eqref{lG_BDD_transform1} and \eqref{lG_BDD_transform2}. The initial values
\[\begin{array}{ll}
 {}_l\S_1(-1,q) = {}_r\S_1(-1,q) =0, 	& {}_l\S_2(-1,q) = {}_r\S_2(-1,q) =0, \\
 {}_l\S_1(0,q) = {}_r\S_1(0,q) =1, 		& {}_l\S_2(0,q) = {}_r\S_2(0,q) =0,\\
 {}_l\S_1(1,q) = {}_r\S_1(1,q) =1+q^2,	& {}_l\S_2(1,q) = {}_r\S_2(1,q) =q, 
\end{array}\] and the recurrence \eqref{lG_BDD_left_Rec} together finish the proof.
\end{proof}

Moreover, The left-hand side of \eqref{lG_BDD_transform1} is the right-hand side of \eqref{lG2_Ali_BDD_EQN} with $x=1$ and $N\mapsto2N$. Also the side by side addition of \eqref{lG_BDD_transform1} and \eqref{lG_BDD_transform2} (with $m\mapsto m-1$) and  using \eqref{Binom_rec} gives the right hand side of \eqref{lG1_Ali_BDD_EQN} with  $x=1$ and $N\mapsto2N$. Hence, we get the following theorem.

\begin{theorem}For any non-negative integer $N$, we have
\begin{align}
\label{Not_AB_lG1_EQN}\sum_{\pi\in\lG_{1,2N}}q^{|\pi|} &=   \sum_{m,n\geq 0} q^{m^2+n^2+n} {N-m\brack n}_{q^2} {n+1\brack m}_{q^2},\\
\label{Not_AB_lG2_EQN}\sum_{\pi\in\lG_{2,2N}}q^{|\pi|} &= \sum_{m,n\geq 0} q^{m^2+n^2+n} {N-m\brack n}_{q^2} {n\brack m}_{q^2}.
\end{align}
\end{theorem}

This theorem is analogous to the one of Berkovich--McCoy--Orrick (Theorem~\ref{GG_Berkovich_McCoy_Orrick_THM}).

It should also be noted that as $N\rightarrow\infty$, \eqref{Not_AB_lG2_EQN} becomes \eqref{lG2_Alladi_Berkovich_EQN} after the use of \eqref{Binom_limit}, change of variables and \eqref{1over_neg_Pochhammer}. The same is not true for the limit of \eqref{Not_AB_lG1_EQN}. The limit $N\rightarrow\infty$ of \eqref{Not_AB_lG1_EQN} after simplifications is \[\sum_{m,n\geq0} \frac{q^{m^2+n^2+n}(1-q^{2n+2})}{(q^2;q^2)_n(q^2;q^2)_{n+1-m}}.\] We do a change of variables with $k=n+1-m$ and compare it with the right-hand side of \eqref{lG1_Alladi_Berkovich_EQN} (with $m$ and $n$ switched followed by $n\mapsto k$). This yields the following corollary.
\begin{corollary}
\begin{equation}\label{simplify_EQN}\sum_{m,k\geq 0} \frac{q^{2m^2+2mk+k^2-m-k}(1-q^{2m+2k})}{(q^2;q^2)_m(q^2;q^2)_k}=\sum_{m,k\geq 0} \frac{q^{2m^2+2mk+k^2-m+k}}{(q^2;q^2)_m(q^2;q^2)_k}.\end{equation}
\end{corollary}

\section{On products of type $(q^a,q^{b};q^M)^{-1}_\infty$}\label{Sec_Products}

In Section~\ref{Sec_Uniform_gap}, we have combinatorially studied minimal gap conditions on partitions. New representations of the generating functions with the gap conditions of Euler's Theorem and the Rogers--Ramanujan identities also fell under this umbrella. Recall that both these theorems \eqref{Euler_Ptt_THM_Combin} and \eqref{RR_THM_Combin} also have a product side. Both these products are of the form \[\frac{1}{(q^s,q^{M-s};q^M)_\infty},\] where the product of Euler's Theorem \eqref{Euler_analytic} is given by $(s,M)=(1,4)$ and the products of Rogers--Ramanujan Theorems \eqref{RR_analytic} are coming from $(s,M)=(1,5)$ and $(2,5)$. So we find it relevant to study the products of this type for positive integers $ s \leq M$.

Let $a,\, b $ and $M$ be positive integers with $M\geq a$ and $M\geq b$. In a more general setting, the product \[\frac{1}{(q^a,q^b;q^M)_\infty}\] is the generating function for the partitions with parts congruent to $a$ or $b$ modulo $M$. We can think of this generating function in a slightly different light. Assume that we start with a sequence $\{a_i\}_{i=0}^\infty$ of positive integers, which consists two interweaving arithmetic progressions $Mj+a$ and $Mj+b$, for $j\geq 0$.  As an end goal, we would like to write a new representation of the generating function \begin{equation}\label{Product_rep_of_P_a_i}\sum_{\pi\in\mathcal{P}_{a_i}} x^{\#(\pi)} q^{|\pi|} = \frac{1}{(xq^a,xq^b;q^M)_\infty},\end{equation} where $\mathcal{P}_{a_i}$ is the set of partitions with parts from the sequence $\{a_i\}$, and $\#(\pi)$ is the number of parts in the partition $\pi$.

Let the smallest element $a_0$ of this sequence be denoted by $s$. Since the sequence is made out of two interlacing sequences there exist two non-negative integers $r$ and $M-r$ such that \[a_{2i+1}-a_{2i}=r\text{   and   }a_{2i+2}-a_{2i+1}=M-r,\] for all $i\geq 0$. Graphically, we can represent the sequence $\{a_i\}$ given in Table~\ref{Chain_diagram_TABLE} 
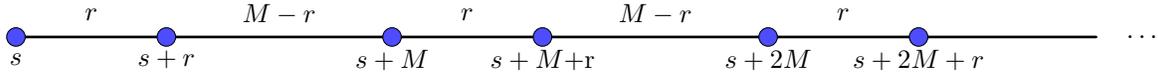
\begin{table}[h]\caption{Graphical path of the parts in the sequence $\{a_i\}$.}\label{Chain_diagram_TABLE}
\definecolor{ududff}{rgb}{0.3,0.3,1.}
\begin{tikzpicture}[line cap=round,line join=round,x=1cm,y=0.6cm]
\clip(0.5,0.) rectangle (16.5,2.);
\draw [line width=1.pt] (1.,1.)-- (15.36,1.);
\draw (1,0.5) node[anchor=center] {$s$};
\draw (3,0.5) node[anchor=center] {$s+r$};
\draw (6,0.5) node[anchor=center] {$s+M$};
\draw (8,0.5) node[anchor=center] {$s+M$+r};
\draw (11,0.5) node[anchor=center] {$s+2M$};
\draw (13,0.5) node[anchor=center] {$s+2M+r$};
\draw (16,1) node[anchor=center] {$\dots$};
\draw (2,1.5) node[anchor=center] {$r$};
\draw (4.5,1.5) node[anchor=center] {$M-r$};
\draw (7,1.5) node[anchor=center] {$r$};
\draw (9.5,1.5) node[anchor=center] {$M-r$};
\draw (12,1.5) node[anchor=center] {$r$};
\begin{scriptsize}
\draw [fill=ududff] (1.,1.) circle (3.5pt);
\draw [fill=ududff] (3.,1.) circle (3.5pt);
\draw [fill=ududff] (6.,1.) circle (3.5pt);
\draw [fill=ududff] (8.,1.) circle (3.5pt);
\draw [fill=ududff] (11.,1.) circle (3.5pt);
\draw [fill=ududff] (13.,1.) circle (3.5pt);
\end{scriptsize}
\end{tikzpicture}
\end{table}

We would like to construct partitions from $\P_{a_i}$ directly. Similar to the previous sections, we start with a minimal configuration and \[\pi_j = (\underbrace{s,s,\dots,s}_{j-\text{terms}}),\]where we fix the number of parts and then we can move forward the terms to  values in the path represented in Table~\ref{Chain_diagram_TABLE}. Therefore, to make any partition of $\P_{a_i}$ into $j$ parts we need a generating function that enumerates partitions into $\leq j$ parts, where each part is either 0 or $r$ modulo $M$. This can be achieved using the following theorem of Berkovich and the author \cite[Th 6.1, p.24]{BerkovichUncu2}. It should be noted that Theorem~\ref{Phi_N_func} was proven using only combinatorial arguments.

\begin{theorem}[Berkovich, U. \cite{BerkovichUncu2}]\label{Phi_N_func} For variables $a,\, b,\, c$, and $d$ and a non-negative integer $N$,
\[ \Phi_N(a,b,c,d):=\frac{1}{(abcd,abcd)_{\lfloor N/2 \rfloor}} \sum_{i=0}^{\lfloor N/2 \rfloor} { \lfloor N/2\rfloor \brack i}_{abcd} \frac{(-a;abcd)_{i+\lceil N/2\rceil -\lfloor N/2 \rfloor} (-abc;abcd)_i}{(ac;abcd)_{i+\lceil N/2\rceil -\lfloor N/2 \rfloor}} (ab)^{\lfloor N/2\rfloor -i}\] is the generating function for the partitions into parts $\leq N$ with Stanley-Boulet weights (where, starting from the largest part, odd-indexed parts are weighted by alternating $a$ and $b$'s, and even-indexed parts are weighted by alternating $c$ and $d$'s).
\end{theorem}

Considering conjugate partitions, $\phi_N(a,b,c,d)$ is the generating function for the partitions into $\leq N$ parts where the odd and even-indexed rows are now weighted with the weights alternating $a$ and $c$'s, and alternating $b$ and $d$'s, respectively. With the choice of variables $(a,b,c,d) = (q^r, q^r, q^{M-r} ,q^{M-r})$ we see that $\Phi_N(q^r, q^r, q^{M-r} ,q^{M-r})$ is the generating function of partitions into $\leq N$ parts where each part is either 0 or $r$ modulo $M$. This is what is required to build the partitions from $\P_{a_i}$ starting from minimal configuration $\pi_j$'s. Given $j$, \[q^{sj}\Phi_j(q^r, q^r, q^{M-r} ,q^{M-r})\] is the generating function for partitions from $\P_{a_i}$ into exactly $j$ parts. Moreover, since we start with fixing the number of parts, it is straightforward to include a variable $x$ to count the number of parts. This way we get a new double sum representation of \eqref{Product_rep_of_P_a_i}:

\begin{theorem}\label{Product_Double_sum_THM}
\begin{align}\nonumber \sum_{\pi \in \P_{a_i}} x^{\#(\pi)} q^{|\pi|} &= \frac{1}{(xq^s,xq^{s+r};q^M)_\infty}\\\label{Product_double_sum_EQN}&= \sum_{j\geq0} \frac{x^j q^{sj}}{(q^{2M},q^{2M})_{\lfloor j/2 \rfloor}} \sum_{i=0}^{\lfloor j/2 \rfloor} { \lfloor j/2\rfloor \brack i}_{q^{2M}} \frac{(-q^r;q^{2M})_{i+\lceil j/2\rceil -\lfloor j/2 \rfloor} (-q^{M+r};q^{2M})_i}{(q^M;q^{2M})_{i+\lceil j/2\rceil -\lfloor j/2 \rfloor}} (q^{2r})^{\lfloor j/2\rfloor -i}. \end{align}
\end{theorem}

Some corollaries of this theorem are as follows:

\begin{corollary}\label{Prod_Corollary} For the choices of the triplets $(s,r,M) = (1,1,2),\ (1,2,4),\ (1,3,5), $ and $(2,1,5)$ the Theorem~\ref{Product_Double_sum_THM} yields
\begin{align}
\label{Prod_GF_Ordinary} \frac{1}{(xq;q)_\infty} &= \sum_{j\geq0} \frac{x^j q^{j}}{(q^{4},q^{4})_{\lfloor j/2 \rfloor}} \sum_{i=0}^{\lfloor j/2 \rfloor} { \lfloor j/2\rfloor \brack i}_{q^{4}} \frac{(-q;q^{4})_{i+\lceil j/2\rceil -\lfloor j/2 \rfloor} (-q^{3};q^{4})_i}{(q^2;q^{4})_{i+\lceil j/2\rceil -\lfloor j/2 \rfloor}} (q^{2})^{\lfloor j/2\rfloor -i},\\
\label{Prod_GF_Distinct}\frac{1}{(xq;q^2)_\infty} &= \sum_{j\geq0} \frac{x^j q^{j}}{(q^{8},q^{8})_{\lfloor j/2 \rfloor}} \sum_{i=0}^{\lfloor j/2 \rfloor} { \lfloor j/2\rfloor \brack i}_{q^{8}} \frac{(-q^2;q^{8})_{i+\lceil j/2\rceil -\lfloor j/2 \rfloor} (-q^{6};q^{8})_i}{(q^4;q^{8})_{i+\lceil j/2\rceil -\lfloor j/2 \rfloor}} (q^{4})^{\lfloor j/2\rfloor -i}.
\end{align}
\end{corollary}

The choices $(s,r,M)=(1,3,5)$ and $(2,1,5)$ proves Theorem~\ref{RRprod_intro_thm}.

Moreover, the double sum that appears in \eqref{Product_double_sum_EQN} calls for an even-odd dissection of the variable $j$. This split and changing the order of summations give us the following relation:

\begin{align}
\nonumber\frac{1}{(xq^s,xq^{s+r};q^M)_\infty}&
= \sum_{j\geq 0} \frac{x^{2j} q^{2sj}}{(q^{2M};q^{2M})_j} \sum_{i=0}^j {j\brack i}_{q^{2M}} \frac{(-q^r,-q^{r+M};q^{2M})_i}{(q^M;q^{2M})_i} q^{2r(j-i)}\\
\label{EODD_sp1}&\hspace{1cm}+xq^s \frac{1+q^r}{1-q^M}\sum_{j\geq 0} \frac{x^{2j} q^{2sj}}{(q^{2M};q^{2M})_j} \sum_{i=0}^j {j\brack i}_{q^{2M}} \frac{(-q^{r+2M},-q^{r+M};q^{2M})_i}{(q^{3M};q^{2M})_i} q^{2r(j-i)}\\
\intertext{After rewriting the $q$-binomial coefficient as $q$-Pochhammer symbols, doing the simple cancellations, changing the order of summation, and shifting $j\mapsto j+i$, we can employ the $q$-exponential sum \cite[p.354, II.1]{Gasper_Rahman} and get}
\nonumber\frac{1}{(xq^s,xq^{s+r};q^M)_\infty}&=\frac{1}{(x^2q^{2(s+r)};q^{2M})}\left[{}_2\phi_1 \left(\begin{array}{c} -q^{r},\, -q^{r+M}\\ q^M\end{array}; q^{2M},x^2 q^{2s}\right)\right.\\\label{EODD_sp2}&\hspace{3cm}+xq^s \frac{1+q^r}{1-q^{M}}\left. {}_2\phi_1 \left(\begin{array}{c} -q^{r+2M},\, -q^{r+M}\\ q^{3M}\end{array}; q^{2M},x^2 q^{2s}\right) \right],
\end{align} where \[{}_r\phi_s \left(\begin{array}{c}a_1,a_2,\dots,a_r\\ b_1,b_2,\dots,b_s\end{array}; q,x\right) = \sum_{n\geq 0 } \frac{(a_1,a_2,\dots,a_r;q)_n}{(q,b_1,b_2,\dots,b_s;q)_n} \left((-1)^n q^{n(n-1)/2}\right)^{1+s-r} x^n.\]

Therefore, by making basic simplifications in \eqref{EODD_sp2} we get the following theorem.

\begin{theorem}\label{JTP_dissection_THM} For $M$ and $s$ positive integers, and $0\leq r \leq M$  we have
\begin{equation}\label{JTP_raw_EQN}\frac{(-xq^{s+r};q^M)_\infty}{(xq^s;q^M)_\infty} = {}_2\phi_1 \left(\begin{array}{c} -q^{r},\, -q^{r+M}\\ q^M\end{array}; q^{2M},x^2 q^{2s}\right) + xq^s \frac{1+q^r}{1-q^{M}} {}_2\phi_1 \left(\begin{array}{c} -q^{r+2M},\, -q^{r+M}\\ q^{3M}\end{array}; q^{2M},x^2 q^{2s}\right).\end{equation}
\end{theorem}

One can also prove Theorem~\ref{JTP_dissection_THM} using the $q$-binomial theorem \cite[p.354, II.3]{Gasper_Rahman} (recalled in Section~\ref{Sec_Uniform_gap}). To that end, one needs to combine the sums on the right-hand side of \eqref{JTP_raw_EQN}. This summation and the even-odd split, \eqref{EODD_sp1}-\eqref{EODD_sp2}, can be considered as a direct hypergeometric proof of the second equality in Theorem~\ref{Product_Double_sum_THM}.  More importantly, if $x=1$ and $M=2s+r$ the equation \eqref{JTP_raw_EQN} leads to the following dissection.

\begin{corollary}\label{JTP_Cor} For positive integer $s$ and non-negative $r$ we have
\begin{align} \nonumber(-q^s,-q^{s+r},q^{2s+r};q^{2s+r})_\infty = (-q^{2(2s+r)+r}&,-q^{2(2s+r)-r},q^{4(2s+r)};q^{4(2s+r)})_\infty \\\label{JTP_ready_eqn} &+ q^s (-q^{r},-q^{4(2s+r)-r},q^{4(2s+r)};q^{4(2s+r)})_\infty .\end{align}
\end{corollary}

Note that Corollary~\ref{JTP_Cor} can be proven by the Jacobi Triple Product identity \cite[p.357, II.28]{Gasper_Rahman} \[\sum_{k=-\infty}^\infty q^{k^2}z^k = (q^2,-zq,-q/z;q^2)_\infty.\] We first write the left-hand side product as a bilateral series using this identity and then do the even-odd dissection of the summation variable. Then employing the Jacobi Triple product identity twice for the dissected series gives \eqref{JTP_ready_eqn}. Here we also give a proof that relies on Theorem~\ref{JTP_dissection_THM}, which we get using Combinatorial tools only, and the $q$-Kummer sum \cite[p.354, II.9]{Gasper_Rahman} \[{}_2\phi_1 \left(\begin{array}{c} a,b \\ aq/b \end{array}; q, -\frac{q}{b}\right) = \frac{(-q;q)_\infty (aq,aq^2/b^2;q^2)_\infty}{(-q/b,aq/b;q)_\infty}.\]

\begin{proof} Let $(x,M) = (1,2s+r)$ in \eqref{JTP_raw_EQN}. Then, $q$-Kummer sum yields the following:
\begin{align}\frac{(-q^{s+r};q^{2s+r})_\infty}{(q^s;q^{2s+r})_\infty} = &
\label{Kummered}\frac{(-q^{2(2s+r)})_\infty (-q^{2(2s+r)+r},-q^{2(2s+r)-r};q^{4(2s+r)})_\infty}{(q^{2s},q^{2s+r};q^{2(2s+r)})_\infty}
\\\nonumber&\hspace{1cm}+ q^s\frac{(-q^{2(2s+r)})_\infty (-q^{r},-q^{4(2s+r)-r};q^{4(2s+r)})_\infty}{(q^{2s},q^{2s+r};q^{2(2s+r)})_\infty}. \end{align}
One only needs to clear the denominators and do simple cancellations to get \eqref{JTP_ready_eqn} from \eqref{Kummered}. 
\end{proof}

\section{Outlook}\label{Sec_Outlook}
Comparing Theorem~\ref{RRprod_intro_thm}, Corollaries~\ref{UG_Corollary} and \ref{Prod_Corollary} gives us intriguing $q$-series identities:

\begin{theorem}\label{Outlook_thm} We have
\begin{align}
\nonumber \sum_{m,n\geq 0} &\frac{x^{m+2n}q^{\frac{m(m+3)}{2}+2n}(1+xq)}{(q;q)_m(q^2;q^2)_n}\\ \label{Ordinary_double_eqn}&= \sum_{j=0} \frac{x^j q^{j}}{(q^{4},q^{4})_{\lfloor j/2 \rfloor}} \sum_{i=0}^{\lfloor j/2 \rfloor} { \lfloor j/2\rfloor \brack i}_{q^{4}} \frac{(-q;q^{4})_{i+\lceil j/2\rceil -\lfloor j/2 \rfloor} (-q^{3};q^{4})_i}{(q^2;q^{4})_{i+\lceil j/2\rceil -\lfloor j/2 \rfloor}} (q^{2})^{\lfloor j/2\rfloor -i},\\
\nonumber\sum_{m,n\geq 0} &\frac{q^{m^2+m+2mn+2n^2+n}(1+q^{1+m+2n})}{(q;q)_m(q^2;q^2)_n} \\\label{Distinct_double_eqn}&= \sum_{j=0} \frac{q^{j}}{(q^{8},q^{8})_{\lfloor j/2 \rfloor}} \sum_{i=0}^{\lfloor j/2 \rfloor} { \lfloor j/2\rfloor \brack i}_{q^{8}} \frac{(-q^2;q^{8})_{i+\lceil j/2\rceil -\lfloor j/2 \rfloor} (-q^{6};q^{8})_i}{(q^4;q^{8})_{i+\lceil j/2\rceil -\lfloor j/2 \rfloor}} (q^{4})^{\lfloor j/2\rfloor -i},\\
\nonumber\sum_{m,n\geq 0} &\frac{q^{\frac{m(3m+1)}{2}+4mn+4n^2}(1+q^{2m+4n+1})}{(q;q)_m(q^2;q^2)_n}\\\label{RR1_double_eqn}&= \sum_{j=0} \frac{ q^{j}}{(q^{10},q^{10})_{\lfloor j/2 \rfloor}} \sum_{i=0}^{\lfloor j/2 \rfloor} { \lfloor j/2\rfloor \brack i}_{q^{10}} \frac{(-q^3;q^{10})_{i+\lceil j/2\rceil -\lfloor j/2 \rfloor} (-q^{8};q^{10})_i}{(q^5;q^{10})_{i+\lceil j/2\rceil -\lfloor j/2 \rfloor}} (q^{6})^{\lfloor j/2\rfloor -i},\\[-1.5ex]\nonumber\\
\nonumber\sum_{m,n\geq 0} &\frac{q^{\frac{m(3m+1)}{2} + 4mn +4n^2 +m+ 2n}(1+q^{2m+4n+2})}{(q;q)_m(q^2;q^2)_n} \\\label{RR2_double_eqn}&= \sum_{j=0} \frac{ q^{2j}}{(q^{10},q^{10})_{\lfloor j/2 \rfloor}} \sum_{i=0}^{\lfloor j/2 \rfloor} { \lfloor j/2\rfloor \brack i}_{q^{10}} \frac{(-q;q^{10})_{i+\lceil j/2\rceil -\lfloor j/2 \rfloor} (-q^{6};q^{10})_i}{(q^5;q^{10})_{i+\lceil j/2\rceil -\lfloor j/2 \rfloor}} (q^{2})^{\lfloor j/2\rfloor -i}.
\end{align}
\end{theorem}

Recall that the left-hand sides are double sum representations of generating functions for the partitions into distinct parts, partitions satisfying the first and the second Rogers--Ramanujan identities, respectively. The right-hand sides correspond to double series representations of partitions into odd parts, partitions into parts $\pm 1$ modulo 5, and $\pm 2$ modulo 5, respectively. Notice that the variable $x$ is only present in the first equality \eqref{Ordinary_double_eqn}, where both sides of the identities are enumerating ordinary partitions and $x$ is counting the number of parts. In the other identities  \eqref{Distinct_double_eqn}-\eqref{RR2_double_eqn}, we need to set $x=1$ as there is not a clear connection regarding the number of parts in Euler's Partition theorem or in the Rogers--Ramanujan Identities.

These equations call for direct proofs. We are hopeful that \eqref{RR1_double_eqn} and \eqref{RR2_double_eqn} lead to new proofs and finite analogues of the Rogers--Ramanujan identities, which might not require the use of the Jacobi Triple Product identity; maybe even one that is purely combinatorial in nature. One thing that needs exploration is attempting to introduce new variables by trial and error in these identities. A generalization of the identities \eqref{RR1_double_eqn} and \eqref{RR2_double_eqn} with a new variable may hint us to a new partition statistic to build a bijective proof on.

On the other hand, on the basic hypergeometric transformations side, this research gave rise to many identities that the author is planning on presenting in detail in his forthcoming paper. One of these such identities is as follows.
\begin{theorem} For $m$ and $n$ non-negative integers, we have
\begin{align*}{m\brack n}_{q^2} {}_2\phi_1&\left( \genfrac{}{}{0pt}{}{q^{-2n},\ q^{-2m+2n}}{ q^{-2m}};q^2,-q \right)\\\ &= q^{n \choose 2}{2m-2n+1\brack n}_{q} {}_4\phi_3\left( \genfrac{}{}{0pt}{}{q^{-n},\ q^{1-n},\ q^{2m-2n+2},\ -q^{2m-2n+2}}{ -q^{-2},\ q^{2m-3n+2},\ q^{2m-3n+3}} ;q^2,q^2 \right).
\end{align*}
\end{theorem}

\section{Acknowledgment}

The author would like to thank George E. Andrews and Christian Krattenthaler for their interest and encouragement. The author also wants to thank the anonymous referees for the careful reading of the manuscript and their valuable suggestions.

The research was partly funded by the Austrian Science Fund (FWF) grant numbers SFB50-07 and SFB50-09, and partly by the EPSRC Grant EP/T015713/1.

\end{document}